\documentclass[12pt,centertags,twoside]{amsart}
\usepackage{amsmath,amstext,amsthm,amscd,typearea,hyperref}
\usepackage{amssymb}
\usepackage{a4wide}
\usepackage{mathrsfs}
\usepackage{typearea}
\usepackage{charter}
\usepackage{pdfsync}
\usepackage[a4paper,width=16.5cm,top=3cm,bottom=3cm]{geometry}

\usepackage{xcolor}

\usepackage[utf8]{inputenc}

\def\1{{\bf 1}}

\def\End{{\rm End}}

\def\dbar{{\bar\partial}}
\def\R{{\mathbb R}}
\def\C{{\mathbb C}}
\def\Z{{\mathbb Z}}
\def\N{{\mathbb N}}

\def\E{{\mathcal E}}

\def\PM{{\mathcal{PM}}}

\def\codim{{\rm codim\,}}
\def\Ok{{\mathcal O}}
\DeclareMathOperator{\Id}{Id}
\DeclareMathOperator{\supp}{supp}
\DeclareMathOperator{\rank}{rank}
\DeclareMathOperator{\im}{im}
\DeclareMathOperator{\coker}{coker}
\DeclareMathOperator{\res}{res}
\DeclareMathOperator{\Hom}{Hom}
\DeclareMathOperator{\id}{id}

\def\V{{\mathcal{V}}}
\def\U{{\mathcal{U}}}

\def\F{{\mathscr{F}}}
\def\G{{\mathscr{G}}}

\def\elem{{e}}

\newcommand{\sing}{\text{\normalfont  sing }}

\newtheorem{thm}{Theorem}[section]
\newtheorem{lma}[thm]{Lemma}

\newtheorem{prop}[thm]{Proposition}

\theoremstyle{definition}

\newtheorem{df}[thm]{Definition}

\theoremstyle{remark}

\newtheorem{preremark}[thm]{Remark}
\newtheorem{preex}[thm]{Example}

\newenvironment{remark}{\begin{preremark}}{\end{preremark}}
\newenvironment{ex}{\begin{preex}}{\end{preex}}

\numberwithin{equation}{section}

\allowdisplaybreaks
\emergencystretch=\maxdimen
\begin{document}

\title[Global Chern currents]{Global Chern currents of coherent sheaves and Baum Bott currents}

\date{\today}

\author[Lucas Kaufmann \& Richard L\"ark\"ang \& Elizabeth Wulcan]
{Lucas Kaufmann \& Richard L\"ark\"ang \& Elizabeth Wulcan}

\address{L.\ Kaufmann \\ Institut Denis Poisson, CNRS, Universit\' e d'Orl\' eans, Rue de Chartres, B.P. 6759, 45067, Orl\' eans cedex 2, FRANCE}

\email{lucas.kaufmann@univ-orleans.fr}

\address{R.\ L\"ark\"ang, E.\ Wulcan \\ Department of Mathematical Sciences\\Chalmers University of Technology and the University of Gothenburg\\SE-412 96
Gothenburg\\SWEDEN}

\email{larkang@chalmers.se, wulcan@chalmers.se}

\begin{abstract}
We provide global extensions of previous results about representations of
characteristic classes of coherent analytic sheaves and of Baum-Bott
residues of holomorphic foliations. We show in the first case that they can be represented by currents with support on the support of the given
coherent analytic sheaf,  and  in the second case,  by currents with support on the singular set of the foliation.
In previous works, we have constructed such representatives provided global resolutions
of the appropriate sheaves existed. In this article, we show that the definition of
Chern classes of Green and the associated techniques, which work on arbitrary complex manifolds without any assumption
on the existence of global resolutions,
may be combined with our previous constructions to yield the desired representatives.

We also prove a transgression formula for such representatives, which is new even in the case when global resolutions exist.
More precisely, the representatives depend on local resolutions of the sheaf, and on choices of metrics and connections on these bundles,
i.e., the currents for two different choices differ by a current of the form $dN$, where $N$ is an explicit current, which in the first case above has support on the support of the
given coherent analytic sheaf, and in the second case above has support on the singular set of the foliation.
\end{abstract}

\maketitle

\section{Introduction}

Let $\G$ be a coherent analytic sheaf on a complex manifold $M$ of dimension $n$.
In case $M$ has the so-called resolution property, i.e., if any coherent analytic sheaf on $M$ admits a resolution  by a complex $(E,\varphi)$
of holomorphic vector bundles of finite length, say $N$, then there is a well-defined notion of Chern forms associated to $\G$, which may be
defined as the alternating product of the Chern and Segre forms associated to connections $D_k$ of the involved vector bundles $E_k$, $k=0,\dots,N$.
This provides an extension of the definition of Chern forms from vector bundles to coherent sheaves, since in case $\G$ is a vector bundle,
and the resolution $E$ is taken to consist of simply $\G$ in degree $0$, the two notions of Chern forms coincide. Furthermore, at the level of
de Rham cohomology classes, the corresponding Chern class so defined is the unique extension of the definition of Chern class from vector bundles
to coherent sheaves which is multiplicative on short exact sequences of vector bundles.

We recall that for example projective manifolds satisfy the resolution property, but also that, as
shown by Voisin, \cite{Voisin}, compact Kähler manifolds do not necessarily satisfy the resolution property. Furthermore, if $M$ is non-compact, the resolution property
always fails, as one may for any discrete sequence $(x_k) \in M$ construct a sheaf $\G$
for which one must have, say, that $\rank E^{x_k}_0 \geq k$ for any resolution
$(E^{x_k},\varphi^{x_k})$ of $\G_{x_k}$.

However,  for any complex manifold $M$ one may find an open cover $\U = (U_\alpha)_{\alpha \in I}$ of $M$ such that
$\G|_{U_\alpha}$ admits a resolution $(E^\alpha,\varphi^\alpha)$ of holomorphic vector bundles
of finite length $N \leq n$. If one equips each complex of vector bundles $E^\alpha$
with a connection $D^\alpha$ (i.e., a tuple of connections $(D^\alpha_N,\dots,D^\alpha_0)$,
where $D^\alpha_k$ is a connection on $E^\alpha_k$), and we fix a partition of unity $(\psi_\alpha)$
subordinate to $\U$, Green showed in \cite{green} that one may associate to this data a Chern form
$c_\ell( (D^\alpha)_{\alpha \in I} )$, which is a $d$-closed smooth form on $M$.  The corresponding de Rham cohomology class only depends on $\G$, and we may denote it as $c_\ell(\G) :=  [c_\ell( (D^\alpha)_{\alpha \in I} )] \in H^{2\ell}_{dR}(M)$. Here $0\leq \ell \leq n$.
In case $M$ has the resolution property, and one takes a single global resolution $(E,\varphi)$
equipped with a connection $D$ and lets
$(E^\alpha,\varphi^\alpha) = (E|_{U_\alpha},\varphi|_{U_\alpha})$ and $D^\alpha = D|_{U_\alpha}$,
then $c_\ell( (D^\alpha)_{\alpha \in I})$ coincides with the Chern form mentioned in the first paragraph.

More generally, for any homogeneous symmetric polynomial $\Phi \in \C[z_1,\ldots,z_n]$,
there is an associated characteristic form $\Phi( (D^\alpha)_{\alpha \in I} )$,
which is a smooth closed form, whose de Rham cohomology class only depends on $\mathcal{G}$.
Recall that the Chern forms are the forms associated to the elementary symmetric
polynomials $\elem_\ell$, $\ell=0,1,\dots,n$.

A coherent sheaf $\G$ will generically be a vector bundle. Let $Z$ denote the
analytic subset of $M$ where $\G$ is not a vector bundle, that is,  the smallest proper analytic subset $Z$ of $M$ such that $\G|_{M \setminus Z}$ is locally free.
In \cite{LW:chern-currents-sheaves} and \cite{KLW}, we showed that associated to
certain coherent sheaves $\G$, \emph{if $\G$ admits a global resolution $(E,\varphi)$},
then one may construct a family of connections $(\widehat{D}^\epsilon)_{\epsilon >0}$ on $E$
such that associated to homogeneous symmetric polynomial $\Phi \in \C[z_1,\ldots,z_n]$
of appropriate degree, the sequence of smooth forms $\Phi(\widehat{D^\epsilon})$ admits a limit
\begin{equation}
    R^\Phi := \lim_{\epsilon \to 0} \Phi(\widehat{D^\epsilon}),
\end{equation}
where $R^\Phi$ is a closed current on $M$ whose de Rham cohomology class represents $\Phi(\G)$.
The key property of the current $R^\Phi$ is that it ``localizes'' $\Phi(\G)$ at the singularities of $\G$,
meaning that $R^\Phi$ has support on $Z$.
In \cite{LW:chern-currents-sheaves}, this is done for sheaves $\G$ whose support $Z$ has codimension $\geq 1$,
and $\Phi$ is of any degree $\geq 1$.
Explicit descriptions are also obtained for $R^\Phi$ when $\deg \Phi \leq \codim Z$.
In \cite{KLW}, this is done for $\G = N\F$ being the normal sheaf of a (singular) holomorphic
foliation $\F$ of $M$, provided the degree of $\Phi$ is larger than the corank of $\F$.
For the compact connected components $Z'$ of $Z$, this also yields representations
of the corresponding so-called Baum-Bott residue of $\F$ along $Z'$, as introduced in
\cite{baum-bott}.

In this article, we show that the constructions from \cite{LW:chern-currents-sheaves} and \cite{KLW} may be generalized to arbitrary complex manifolds $M$,
i.e., that we may drop the assumption of $\G$ having a global resolution $(E,\varphi)$.

Recall that the \emph{(fundamental) cycle} of $\G$ is the cycle
\begin{equation}\label{eq:fundcycle}
    [\G] = \sum_k m_k [Z_k]
\end{equation}
(considered as an integration current), where $Z_k$ are the
irreducible components of $\supp \G$, and $m_k$ is the \emph{geometric multiplicity}
of $Z_k$ in $\G$, see \emph{e.g.}\ \cite[Chapter~1.5]{Fulton}.

Our global generalization of the results in \cite{LW:chern-currents-sheaves} is the following.

\begin{thm} \label{thm:sheavesBasic}
    Let $M$ be a complex manifold of dimension $n$, let $\G$ be a coherent analytic sheaf on $M$,
    such that $Z = \supp \G$ has codimension $p \geq 1$ and let $\Phi$ be a symmetric
    polynomial $\Phi \in \C[z_1,\ldots,z_n]$ of degree $\ell$ with $1 \leq \ell \leq n$.
    Let $(U_\alpha)_{\alpha\in I}$ be a Stein open cover of $M$, and let $N \in \mathbb{N}$ be such that each $\G|_{U_\alpha}$ admits a finite resolution
    $(E^\alpha,\varphi^\alpha)$ of holomorphic vector bundles of length $\leq N$.
    Assume that each $E^\alpha$ is equipped with Hermitian metrics and a connection
    $D^\alpha$ of type $(1,0)$.
    Then, for $\epsilon >0$,  there are explicit connections $\widehat{D}^{\alpha,\epsilon}$ on $E^\alpha$ such that
    the limit
    \begin{equation*}
        R^\Phi := \lim_{\epsilon \to 0} \Phi( (\widehat{D}^{\alpha,\epsilon} )_{\alpha \in I})
    \end{equation*}
    exists as a current, which represents $\Phi(\mathcal{G})$, and which has support on $Z$.
    In addition, if $Z$ has pure codimension $p$, then
    \begin{equation} \label{eq:fundCycle}
        R^{\elem_p} = (-1)^{p-1}(p-1)! [\G].
    \end{equation}
    If $\ell < p$, then
    \begin{equation} \label{eq:vanish1}
        R^{\elem_\ell} = 0,
    \end{equation}
    and if $\ell_1+\cdots+\ell_m \leq p$, where $m \geq 2$, then
    \begin{equation} \label{eq:vanish2}
        R^{\elem_{\ell_1} \cdots \elem_{\ell_m}} = 0.
    \end{equation}
\end{thm}

The connections $\widehat{D}^{\alpha,\epsilon}$ depend on the choice of connections $D^\alpha$,
as well as Hermitian metrics on $E^\alpha$ (where the connection and the metrics do
not need to be related to each other) and a partition of unity of the cover.
In general, the resulting current $R^\Phi$ will also depend on these choices.
However, as follows by \eqref{eq:fundCycle},\eqref{eq:vanish1} and \eqref{eq:vanish2},
$R^\Phi$ is independent of all these choices if $\deg \Phi \leq \codim \supp \G$.
In general, we prove a transgression formula which says that they are independent of these choices up
to a current of the form $dN^\Phi$, where $N^\Phi$ is a current with support on $Z$, see Theorem~\ref{pizza}.
This last part is new, also in the case when $\G$ admits a global resolution.

In particular, at the level of cohomology,
\begin{equation}\label{janakippo}
c_p(\mathcal \G) = (-1)^{p-1}(p-1)![\mathcal \G],
  \end{equation}
where now the right hand side should be interpreted as a de Rham class.
If for example $\G$ is the pushforward of a vector bundle from a submanifold,
the fact that \eqref{janakippo} holds is a well-known consequence of the Grothendieck-Riemann-Roch
theorem, \cite{OBTT}, \emph{cf.}\ \cite[Examples~15.2.16 and~15.1.2]{Fulton}.

Just as in \cite{LW:chern-currents-sheaves}, in case $\G$ has codimension $p$,
but not necessarily \emph{pure} codimension $p$, then Theorem~\ref{thm:sheavesBasic}
still holds if we  replace the first equation by
\begin{equation} \label{orenstrumpa}
        R^{e_p} = (-1)^{p-1}(p-1)![\G]_p,
\end{equation}
where $[\G]_p$ denotes the part of $[\G]$ of codimension $p$,
\textit{i.e.} in \eqref{eq:fundcycle}, one only sums over the components $Z_i$ of codimension $p$.

\bigskip

Let now $\F$ be a holomorphic foliation of rank $\kappa$ on $M$ and denote by $N \F$ its
normal sheaf, see Section ~\ref{subsec:prelim-foliations} for the
definitions. If $Z'$ is a compact connected component of the singular set of $\F$, $\sing \F$,
and $\Phi \in \C[z_1,\ldots,z_n]$ is a homogeneous symmetric polynomial of degree $\ell$
with $n-\kappa < \ell \leq n$, then Baum-Bott, \cite{baum-bott}, defined a class
$\res^\Phi(\F;Z')$, an object in $H_{2n-2\ell}(Z',\C)$, which may also be represented as a
de Rham cohomology class on $M$ with compact support, whose representatives  have support in arbitrarily small neighborhoods of $Z'$.
Provided $\sing N\F$ is compact, these classes together represent $\Phi(N\F)$, i.e.,
\begin{equation*}
 \sum_{Z' \subset \sing \F} \res^\Phi(\F;Z') = \Phi (N \F) \quad \text{ in } \quad H^{2\ell}(M,\C),
\end{equation*}
where the sum is over all the connected components of $Z'$. This should be seen as
a localization formula for $\Phi (N \F)$ around the singularities of $\F$.

We may assume that $M$ has a Stein open cover $(U_\alpha)$ such that
for each $\alpha$, $N\F|_{U_\alpha}$ admits a resolution
\begin{equation} \label{anebrun}
    0 \to E_N^\alpha \stackrel{\varphi^\alpha_N}{\longrightarrow} E_{N-1}^\alpha \stackrel{\varphi^\alpha_{N-1}}{\longrightarrow} \dots \stackrel{\varphi_2}{\longrightarrow} E_1^\alpha \stackrel{\varphi^\alpha_1}{\longrightarrow} E_0^\alpha = TM
    \stackrel{\varphi^\alpha_0}{\longrightarrow} N\F \to 0.
\end{equation}
By the syzygy theorem, one may assume that $N \leq n$.

The global version of the results in \cite{KLW} that we obtain is the following.

\begin{thm} \label{thm:foliationsBasic}
 Let $M$ be a complex manifold of dimension $n$,  let $\F$ be a
 holomorphic foliation of rank $\kappa$ on $M$, and let $\Phi\in
 \C[z_1,\ldots, z_n]$ be a homogeneous symmetric polynomial of degree
 $\ell$ with $n-\kappa<\ell\leq n$.
 Let $(U_\alpha)$ be a Stein open cover of $M$, and let $N \in \mathbb{N}$ be such that each
 $N\F|_{U_\alpha}$ admits resolutions of the form \eqref{anebrun}.
 Assume that each $TM$ is equipped with a torsion free connection $D^{TM}$ of type $(1,0)$,
 and that each $E^\alpha_k$ is equipped with a connection $D^\alpha_k$ for $k=1,\dots,N$.
Then, for $\epsilon >0$ there are explicit connections $\widehat{D}^{\alpha,\epsilon}$ on $E^\alpha$ such that the limit
\begin{equation*}
    R^\Phi := \lim_{\epsilon \to 0} \Phi( (\widehat{D}^{\alpha,\epsilon} )_{\alpha \in I})
\end{equation*}
exists as a current, which represents $\Phi(N\F)$, and which has support on $\sing \F$.

Write $R^\Phi = \sum R^\Phi_{Z'}$, where the sum runs over all connected
components of $\sing N\F$, and $R^\Phi_{Z'}$ has support on $Z'$.
If $Z'$ is a compact connected component of $\sing \F$, then $R^\Phi_{Z'}$
represents the Baum-Bott residue $\res^\Phi(\F;Z') \in H_{2n-2\ell}(Z',\C)$.
Furthermore, if $\ell = \codim Z'$, then $R^\Phi_{Z'}$
only depends on $\F$, and if $\ell < \codim Z'$, then $R^\Phi_{Z'} = 0$.
\end{thm}

As in the previous theorem, the connections and resulting currents $R^\Phi$ and $R^\Phi_{Z'}$ in
general depend on the choice of connections $D^\alpha$, as well as Hermitian metrics on
$E^\alpha$ and a partition of unity of the cover. However, we prove a transgression formula which says
that they are also independent of these choices up to a current of the form $dN^\Phi$, where $N^\Phi$ is a current with
support on $\sing \F$, see Theorem~\ref{pizza}.
Also in this case, this last extends the previous results when $\G$ admits a global resolution.
In \cite{KLW}, we only proved such a result when the resolution was fixed, but possibly
equipped with different sets of connections and metrics.

\bigskip

Chern classes of coherent sheaves, without the assumption of the existence of a global locally free resolution,
has been studied in various recent papers, including \cite{Gri,Qiang,BSW,Wu}, in finer cohomology theories than
de Rham cohomology, more precisely in (rational or complex) Bott-Chern and Deligne cohomology.
In the present paper, our focus has been to find explicit representatives of Chern classes of a coherent sheaf
with support on the support of the sheaf (or in its singular support in the foliation case),  a type of result which as far as we can tell, none of the above
mentioned works seems to consider. Our methods do not seem to yield representatives in the
finer cohomology theories mentioned above, as for example our construction is based on Chern forms of connections
that are not Chern connections of a hermitian metric.

\bigskip

At the final stages of the preparation of the present article appeared a preprint by Han, \cite{Han}, with similar results as our Theorem~\ref{thm:sheavesBasic},
i.e., the main result states that there exists representations of characteristic classes of $\F$ as pseudomeromorphic currents with support on $Z = \supp \F$.
Both in Han's work and in the present article,  the approach is based on the local construction from \cite{LW:chern-currents-sheaves}.  In \cite{Han} however, the globalization  builds on a theory of global resolutions of coherent sheaves by cohesive modules due to Block, \cite{Block}, and the associated characteristic forms as introduced by Qiang,
\cite{Qiang}, cf., also \cite{BSW},  while here we adopt Green's approach.



\section{Coherent sheaves, holomorphic foliations, vector bundle complexes and characteristic classes}\label{sec:prelim}

Throughout the paper $M$ will be a complex manifold of dimension $n$.

\subsection{Holomorphic foliations} \label{subsec:prelim-foliations}

A \textit{holomorphic foliation}  $\F$ on $M$ is the data of a coherent analytic
subsheaf $T \F$ of $TM$, called the \textit{tangent sheaf} of $\F$, such
that $T \F$ is \emph{involutive}, that is, for any pair of local sections
$u,v$ of $T \F$, the Lie bracket $[u,v]$ belongs to $T \F$.
The \textit{normal sheaf} of $\F$ is $N \F:=TM / T\F$.

The generic rank of $T \F$ is called the \textit{rank of $\F$}. Note
that $N\F$ is a coherent analytic sheaf. The \textit{singular set} of
$\F$ is, by definition, the smallest subset $\sing \F \subset M$
outside of which $N \F$ is locally free. We say that $\F$ is
\textit{regular} if $\sing \F$ is empty.  By definition, the
restriction of  $N\F$ to $M \setminus \sing \F$ defines a regular foliation
whose normal sheaf is a holomorphic vector bundle of rank $n-\kappa$,
where $\kappa$ is the rank of $\F$. Moreover, by Frobenius Theorem, over
$M \setminus \sing \F$, the bundle $T\F$ is locally given by vectors
tangent to the fibers of a (local) holomorphic submersion.

\smallskip

It is standard in the literature to also assume that a foliation $\F$ is
\emph{saturated} or \emph{full}, that is, $N \F:=TM / T\F$ is assumed to
be torsion free. In that case, one avoids certain ``artificial'' singularities
which for example could be caused by having a generator of $T\F$ where all the entries
share a common factor. We do not make use of this condition in our proofs, so in this
article, we do not add this assumption.

\subsection{Vector bundle complexes, connections, and superstructure}\label{superduper}
Consider a vector bundle complex
\begin{equation} \label{eq:vbComplex}
     0 \to E_N \xrightarrow[]{\varphi_N}  \cdots\xrightarrow[]{\varphi_1}  E_0
\end{equation}
   over $M$.
Following \cite{andersson-wulcan:ens}, we equip $E := \bigoplus_{k=0}^N E_k$ with a superstructure by letting $E^+ =  \bigoplus_{2k} E_k$ (resp.\ $E^- =  \bigoplus_{2k + 1} E_k$) be the even (resp.\ odd) parts of $E = E^+ \oplus E^-$. This is a $\Z \slash 2 \Z$-grading that simplifies some of the formulas and computations. An endomorphism $\varphi \in \End(E)$ is even (resp.\ odd) if it preserves (resp.\ switches) the $\pm$-components.

The superstructure affects how form-valued endomorphisms act.
Assume that $\alpha=\omega\otimes \gamma$ is a form-valued section of $\End (E)$, where $\omega$ is a smooth form of degree $m$ and $\gamma$ is a section of
$\Hom(E_\ell,E_k)$.
We let $\deg_f \alpha = m$ and $\deg_e \alpha = k-\ell$ denote
the \emph{form} and \emph{endomorphism degrees}, respectively, of $\alpha$.
The \emph{total degree} is $\deg \alpha = \deg_f \alpha + \deg_e \alpha$.
If $\beta$ is a form-valued section of $E$, i.e., $\beta=\eta\otimes\xi$, where
$\eta$ is a smooth form, and $\xi$ is a section of $E$, both homogeneous in degree, then
the we define the action of $\alpha$ on $\beta$ by
\begin{equation}\label{thyme}
    \alpha(\beta) := (-1)^{(\deg_e \alpha)(\deg_f \beta)} \omega \wedge \eta \otimes \gamma(\xi).
\end{equation}
If furthermore, $\alpha'=\omega'\otimes \gamma'$, where $\gamma'$ is a holomorphic
section of $\End (E)$, and $\omega'$ is a smooth form, both homogeneous in degree, then
we define
\begin{equation}\label{rosemary}
    \alpha \alpha' := (-1)^{(\deg_e \alpha)(\deg_f \alpha')} \omega\wedge \omega' \otimes \gamma \circ \gamma'.
  \end{equation}

  \smallskip

Assume that each $E_k$ is equipped with a connection $D_k$. Then there
is an induced connection $D_E$ on $E$, that in turn induces a connection $D_{\End}$ on $\End (E)$,
that takes the superstructure into account, which is defined by
\begin{equation*}
D_\End\alpha = D_E\circ \alpha - (-1)^{\deg \alpha}\alpha \circ D_E.
  \end{equation*}
We will often denote $D_{\End}$ simply by $D$.

Similarly to \cite{baum-bott} we say that the collection of connections $(D_N, \ldots, D_0)$
is \emph{compatible} with \eqref{eq:vbComplex} if $D\varphi_k = 0$ for $k=1,\dots,N$.
If $E_{-1} := \coker \varphi_1$ is a vector bundle, one may consider the augmented complex
\begin{equation} \label{eq:augmentedEcomplex}
     0 \to E_N \xrightarrow[]{\varphi_N}  \cdots\xrightarrow[]{\varphi_1}  E_0\xrightarrow[]{\varphi_0} E_{-1} \to 0.
\end{equation}
Then, one may say that the collection of connections $(D_N, \ldots, D_{-1})$ is
compatible with \eqref{eq:augmentedEcomplex} if $D\varphi_k = 0$ for $k=0,\dots,N$,
cf., i.e., \cite[Section~2.2]{KLW} and \cite[Definition~4.16]{baum-bott}.
Note that if \eqref{eq:vbComplex} is equipped with connections $(D_N,\ldots,D_0)$
compatible with \eqref{eq:vbComplex}, then there is an induced connection
$D_{-1}$ on $E_{-1}$ which makes $(D_N,\ldots,D_{-1})$ compatible with
\eqref{eq:augmentedEcomplex}, and conversely,
such a set of connections makes $(D_N,\ldots,D_0)$ compatible with \eqref{eq:vbComplex},
and $D_{-1}$ to be the induced connection on $E_{-1}$.

\subsection{Characteristic classes and forms}\label{texas}

We begin by briefly reviewing the definition of characteristic classes and
forms of coherent sheaves which admit global locally free resolutions.
In Section~\ref{sec:globalChern}, we will discuss how the definition may be
extended to the case when global resolutions do not exist. Most of the material
in this section can be found in \cite[Sections~1 and ~4]{baum-bott}.

Let $E$ be a smooth complex vector bundle of rank $r$ over $M$ equipped with a connection $D$.
We define $2j$-forms $\elem_j(D)$ associated to $(E,D)$ through
\begin{equation}\label{butterfly}
\det \left (I + \frac{i}{2\pi}\Theta(D) \right) = 1 + \elem_1(D) + \cdots + \elem_n(D).
\end{equation}
Note that $\elem_j(D)$ is just the $j$th Chern form of $(E,D)$.

Consider now a complex $(E,\varphi)$ of smooth vector bundles
\begin{equation}\label{snow2}
  0 \to E_N \to \cdots \to E_0,
  \end{equation}
and assume that the vector bundles in \eqref{snow2} are equipped with connections
$D_0, \ldots, D_N$, and let $\Theta(D_k)$, $k=0 ,\ldots, N$, be the corresponding
curvature forms. We will sometimes use short-hand notation, $D=(D_N,\dots,D_0)$
and say that $E$ is equipped with a connection $D$.
Generalizing \eqref{butterfly} we let $\elem_j(D)$ be the $2 j$-form
defined by
\begin{equation} \label{eq:mixed-chern}
\prod_{k=0}^N \left( \det \left[I + \frac{i}{2\pi}\Theta(D_k) \right] \right)^{(-1)^k}  =
\elem(D) = 1 + \elem_1(D) + \cdots + \elem_n(D).
\end{equation}

Let $\Phi \in \C[z_1,\ldots,z_n]$ be a homogeneous symmetric polynomial of degree
$\ell \leq n$. In case $\Phi=\elem_j$, we have defined $\elem_j(D)$ above.
In general, there is a unique polynomial $\widehat \Phi$, such that
$$
\Phi(z_1,\ldots, z_n)=\widehat \Phi\big (\elem_1(z), \ldots, \elem_\ell(z)\big ),
$$
and $\Phi(D)$ is then be defined as
\begin{equation}\label{schmetterling}
  \Phi(D) =  \widehat \Phi\big (\elem_1(D), \ldots, \elem_\ell(D)\big ).
\end{equation}

One motivation for the definition \eqref{schmetterling} at the level of forms is the following result:
\begin{lma}[\cite{baum-bott} - Lemma 4.22]  \label{lemma:exact-curvature}
  Assume that the complex \eqref{eq:vbComplex} is pointwise exact and that $E_1 = \coker \varphi_1$
  is a vector bundle over some open set $U \subset M$. If $D=(D_N, \ldots, D_0)$ is compatible
  with $(E,\varphi)$, and $D_{-1}$ is the induced connection on $E_{-1}$, then
 \begin{equation*}
\Phi(D) = \Phi(D_{-1}) \quad \text{on} \quad U.
\end{equation*}
\end{lma}

\medskip

Next, let $\G$ be a coherent analytic sheaf over $M$ and assume that $\G$ admits a resolution
of the form \eqref{snow2}, i.e.,
\begin{equation}\label{snow}
  0 \to E_N \to \cdots \to E_0 \to \G \to 0.
\end{equation}
is exact, and equip $E$ with some connection $D$. The $j$th Chern class of $\G$ is defined as
$c_j(\G) = [e_j(D)] \in H^{2j}_{dR}(M)$, where $H^{\bullet}_{dR}(M)$ denotes de Rham cohomology
(of smooth forms or currents).

In case $M$ has the \emph{resolution property}, i.e., if any coherent
sheaf $\G$ admits a resolution \eqref{snow}, then it is classical that $[e_j(D)]$ is a well-defined
class, i.e., it is independent of the choice of resolution and connection, and it is
the unique extension of the definition of Chern classes from vector bundles to coherent sheaves
which is multiplicative on short exact sequences, i.e., if
$
    0 \to \G'' \to \G \to \G' \to 0
$
is an exact sequence of coherent analytic sheaves, then
$
    c(\G) = c(\G'') c(\G').
$

It follows more generally by the construction of Green of Chern classes of coherent sheaves that we discuss later, that $c_j(\G) = [e_j(D)]$ is well-defined even if $M$ does not have the resolution property, provided that a resolution \eqref{snow} exists for $\G$.

If $\Phi \in \C[z_1,\ldots,z_n]$, $\G$, $E$ and $D$ are as above,
then $\Phi(\G)\in H^{2\ell}(M, \C)$ is defined as $\Phi(\G) = [\Phi(D)]$.

\subsection{Elementary sequences}\label{apple}

A vector bundle complex $(E, \varphi)$ is said to be an
\emph{elementary sequence} in the vector bundles $E^1, \ldots, E^L$ if
it is a direct sum of trivial complexes
\begin{equation} \label{trivial}
    (0 \to E^j \stackrel{\id}{\to} E^j \to 0)[m],
\end{equation}
where $j\in \{1, \ldots, L\}$, $m \in \Z$, and $[m]$ denotes the complex shifted in degree $m$ steps.

Note that the pullback by a smooth map of an elementary sequence is again an elementary sequence.
Assume that each $E^j$ is equipped with a connection $D^j$ and let $D$
be the induced connection on $(E, \varphi)$. Note that then $D$ is
compatible with $(E, \varphi)$.

\begin{lma}
    \label{lma:phiElementarySequence}
    Let $(E,\varphi)$ be complex of vector bundles as in \eqref{snow}, and
    let $(F,\eta)$ be an elementary sequence in some vector bundles $F^1,\dots,F^L$.
    Assume that $(E,\varphi)$ is equipped with a connection $D^E$,
    and that $(F,\psi)$ is equipped with a connection $D^F$ induced by a set of
    connections $D^{F^1},\dots,D^{F^L}$ on $F^1,\dots,F^L$.
    Let $\Phi\in \C[z_1, \ldots, z_n]$ be a homogeneous symmetric polynomial.
    Then
    \begin{equation} \label{eq:elementarySequenceInvariant}
        \Phi(D^E\oplus D^F) =
        \Phi(D^E).
    \end{equation}
\end{lma}

\begin{proof}
It follows by \eqref{eq:mixed-chern} and the multiplicative property of the determinant on block
diagonal matrices that
\begin{equation*}
\elem(D^E\oplus D^F)
    =
\elem(D^E) \elem(D^F).
\end{equation*}
Hence, \eqref{eq:elementarySequenceInvariant} holds for $\Phi$ being any of the $e_\ell$,
$\ell=1,\dots,n$ by Lemma~\ref{lemma:exact-curvature}, since $(F,\eta)$ is a resolution of
the sheaf $0$ and is equipped with a compatible connection.
In then follows that \eqref{eq:elementarySequenceInvariant} holds for general $\Phi$
by \eqref{schmetterling}.
\end{proof}

\subsection{Interpolation of connections}

Consider a locally free resolution $(E,\varphi)$ of $\mathcal{G}$ of length $N$ as in \eqref{snow}.
Assume that $E$ is equipped with a set of connections $D^0,\ldots,D^p$, i.e.,
$D^j=(D_N^j,\ldots,D_0^j)$, where $D_k^j$ is a connection on $E_k$ for
$k=0,\ldots,N$, $j=0,\ldots,p$.
Let
\begin{equation}\label{snowfall}
  \Delta_p = \big \{ (t_0,\dots,t_p) \in [0,1]^p \mid t_0+\dots+t_p =
  1\big \}
\end{equation}
be the standard $p$-simplex, with the orientation determined by the form $dt_1\wedge\dots\wedge dt_n$,
and let $\pi$ denote the projection $M\otimes \Delta_p \to M$.
Define a connection $D$ on $\pi^* E$ by
\begin{equation*}
D_k = \sum_{j=0}^p t_j \pi^* D^j_k.
\end{equation*}
Let $\Phi\in \C[z_1, \ldots, z_n]$ be a homogeneous symmetric
polynomial of degree $\ell$.
We define
\begin{equation} \label{eq:PhiD0Dp}
    \Phi(D^0,\ldots,D^p) := \pi_* \Phi(D).
\end{equation}
By \cite[equation~(II.8.2)]{suwa:book},
\begin{equation} \label{eq:suwa-interpolation}
    \sum_{k=0}^p (-1)^k \Phi(D^0,\ldots,\widehat{D^k},\ldots,D^p) + (-1)^p d\Phi(D^0,\ldots,D^p) = 0,
\end{equation}
where $\widehat{D^k}$ denotes that $D^k$ is removed.

\subsection{The \v Cech-de Rham complex}

Given a locally finite open cover $\U= (U_\alpha)_{\alpha\in I}$ of $M$, let $N\U$ denote
the nerve of $\U$. Throughout we will tacitly assume that all covers
are locally finite.
For $\Delta=(\alpha_0,\ldots, \alpha_p)\in N\U$, let
\[
  U_\Delta = U_{\alpha_0, \ldots,
    \alpha_p}:=U_{\alpha_0}\cap\cdots\cap U_{\alpha_p}.
\]

Let
\[
  \check C^\bullet(\U, \E^\bullet) = \bigoplus_{p,q\geq 0} \check
  C^p(\U, \E^q)
\]
be the (non-alternating) \v{C}ech-de Rham complex,
where $\check C^p(\U, \E^\bullet)$ denote the $q$-form-valued \v Cech $p$-cochains
\[
\gamma=(\gamma_{\alpha_0,\ldots, \alpha_p}) \in \prod_{(\alpha_0,\ldots, \alpha_p) \in I^{p+1}}
\E^q(U_{\alpha_0,\ldots, \alpha_p}),
\]
and $\E^q$ denotes the sheaf of smooth $q$-forms.
We say that
$\gamma \in \check  C^p(\U, \E^q)$ has \v{C}ech degree $p$, form degree $q$, and total degree
$p+q$.
If each $\gamma_{\alpha_0,\ldots, \alpha_p}$ has bidegree $(r,s)$ we
will sometimes use the notation $\gamma\in \check C^p(\U, \E^{r,s})$.
Let $\delta=\delta^p: \check C^p(\U, \E^\bullet) \to \check C^{p+1}(\U,
\E^\bullet)$ be the \v Cech differential
\[
  (\delta \gamma)_{\alpha_0, \ldots, \alpha_{p+1}} =
  \sum_{0\leq  j\leq p+1}(-1)^j \gamma_{\alpha_0, \ldots \widehat \alpha_j
  \ldots, \alpha_{p+1}}|_{U_{\alpha_0, \ldots, \alpha_{p+1}}},
\]
and let $d=d^p: \check C^p(\U, \E^q) \to \check C^{p}(\U,\E^{q+1})$ be the exterior derivative
\[
    (d \gamma)_{\alpha_0,\ldots,\alpha_p} = (-1)^{p}d\gamma_{\alpha_0,\ldots,\alpha_p}.
\]
Let
\[
    \nabla=d+\delta
\]
be the (total) differential of the \v{C}ech-de Rham complex.

Recall that $H^k(\check{C}^\bullet(\U,\E^\bullet)) \cong
H^k_{dR}(M)$ and that one may construct an explicit morphism of complexes
$\Psi : \check{C}^\bullet(\U,\E^\bullet) \to \E^\bullet(M)$ which induces this isomorphism, cf., i.e.,
\cite[Proposition~9.5]{BT}\footnote{In \cite{BT}, they consider the alternating \v{C}ech-de Rham complex. However, the proof of \cite[Proposition~9.5]{BT} works
equally well for the non-alternating \v{C}ech-de Rham complex considered in this article.}. Assume that $(\psi_\alpha)_{\alpha\in I}$ is a partition of unity subordinate to
$\U$, and define $\psi : \check{C}^p(\U,\E^\bullet) \to \check{C}^{p-1}(\U,\E^\bullet)$ for $p \geq 1$ by
\begin{equation*}
    (\psi \gamma)_{\alpha_0,\dots,\alpha_{p-1}} := -\sum_{\alpha \in I} \psi_{\alpha}\gamma_{\alpha_0,\dots,\alpha_{p-1},\alpha}.
\end{equation*}
If $\gamma \in \check{C}^\bullet(\U,\E^\bullet)$, and $\gamma_p$ denotes the part of $\gamma$ of \v{C}ech degree $p$, then we let
\begin{equation}\label{cream'}
    \Psi'\gamma := \sum_{p \geq 0} (d\psi)^p \gamma_p + \psi (d\psi)^p (D\gamma)_{p+1},
\end{equation}
where $(d\psi)^p$ denotes the composition $d\circ \psi$ repeated $p$ times.
Note that if $\gamma$ is a cocycle, then
\begin{equation}\label{cream2}
    \Psi'\gamma = \sum_{p \geq 0} (d\psi)^p \gamma_p.
\end{equation}
The cochain $\Psi'\gamma$ defined by \eqref{cream'} lies in $\check{C}^0(\U,\E^\bullet)$, but as shown in
\cite[Proposition~9.5]{BT}, it is in fact $\delta$-closed, so it induces a global form in $\E^\bullet(M)$.
We denote this global form by $\Psi(\gamma)$, and note that it may thus either be defined
by $\Psi(\gamma)|_{U_\alpha} = \Psi'(\gamma)_\alpha$, or by
\begin{equation} \label{cream}
    \Psi\gamma := \sum \psi_\alpha (\Psi'\gamma)_\alpha.
\end{equation}

Assume that $\V=(V_\beta)_{\beta\in J}$ is an open cover which refines $\U$,
i.e., that there is a \emph{refinement map} $\rho: J\to I$ such that $V_\beta\subset
U_{\rho(\beta)}$ for all $\beta\in J$.
Then there
are induced morphisms
$\rho=\rho^p :\check C^p(\U, \E^\bullet) \to \check C^p(\V, \E^\bullet)$,
defined by
\[
  (\rho \gamma)_{\beta_0,\ldots, \beta_p} =
  \gamma_{\rho(\beta_0),\ldots, \rho(\beta_p)}|_{V_{\beta_0, \ldots,
      \beta_p}}.
  \]
  Assume that $\rho_1$ and $\rho_2$ are two refinement maps and let
  $h=h^p: \check C^p(\U, \E^\bullet) \to \check C^{p-1}(\V,
  \E^\bullet)$ be the homotopy operator defined by
  \begin{equation}\label{semla}
    (h \gamma)_{\beta_0,\ldots, \beta_{p-1}} =
    \sum_{k=0}^{p-1}
    (-1)^k \gamma_{\rho_1(\beta_0),\ldots, \rho_1(\beta_k), \rho_2(\beta_k),\ldots, \rho_2(\beta_{p-1})}
    |_{V_{\beta_0, \ldots,\beta_{p-1}}}.
\end{equation}
A calculation yields that
\begin{equation}\label{almond}
  \nabla h+h\nabla=\rho_2-\rho_1.
\end{equation}

\section{Baum-Bott theory}\label{bbres}

The theory of Baum-Bott residues was developed in \cite{baum-bott},
extending the theory of rank one foliations in \cite{baum-bott-mero}
to general foliations.
The main outcome of Baum-Bott's theory is the fact that
high degree characteristic classes $\Phi(N\F)$ of $N \F$ localize around $\sing \F$, cf.\ the introduction.
This is a consequence of a vanishing theorem for the normal bundle of a regular foliation due to the existence of special connections.
Recall that a connection is said to be of \emph{type $(1, 0)$}, or a \emph{$(1,0)$-connection}, if its $(0, 1)$-part equals $\dbar$.

\begin{df} \label{def:basic-connection} (\cite {baum-bott} - Definition 3.24) Let $\F$ be a regular foliation on $M$ and let $\varphi_0: TM \to  N\F$ be the canonical surjection. A connection $D$ on $N \F$ is \textit{basic} if it is of type $(1,0)$ and
\begin{equation} \label{eq:basicCondition}
i(u)D( \varphi_0 v) =  \varphi_0 [u,v]
\end{equation}
for any smooth sections $u$ of $T \F$ and $v$ of $TM$.
\end{df}

It is not hard to see that basic connections always exist, see
\cite[\S 3]{baum-bott}.

\begin{thm}[Baum-Bott's Vanishing theorem,
  \cite{baum-bott} - Proposition ~3.27] \label{thm:vanishing}
  Let $\F$ be a regular foliation of rank $\kappa$ on a complex
  manifold $M$ of dimension $n$.
  If $D$ is a basic connection on $N \F$, then $\Phi(D) = 0$
 for every homogeneous symmetric polynomial $\Phi \in \C[z_1,\ldots,z_n]$ of degree $\ell$ with $n-\kappa < \ell \leq n$.
\end{thm}

\subsection{Baum-Bott residues}\label{lucia}

In the presence of singularities, one cannot work directly with connections on $N\F$, so the use of suitable resolutions is necessary.

Let $Z'$ be a compact connected component of $\sing \F$. Then one can
find an open neighborhood $U$ of $Z'$ in $M$ such that $U\cap \sing
\F=Z'$ and $Z'$ is a deformation retract of $U$, and a locally free resolution
$(E,\varphi)$ of $N\F$ of $\mathcal A$-modules on $U$, where $\mathcal A$ denotes
the sheaf of germs of real analytic functions, cf.\ \cite[Proposition~6.3]{baum-bott}.

Given a basic connection $D_{-1}$ on $N\F|_{M\setminus \sing \F}$,
and a compact neighborhood $\Sigma \subset U$ of $Z'$,
as is shown in \cite[Lemma~5.26]{baum-bott}, one may construct a connection
$D$ on $E$ which is compatible with $(E,\varphi)$ on $U\setminus \Sigma$,
and such that $D_{-1}$ is the connection on $N\F|_{U\setminus \Sigma}$
induced by $D$.
If $\Phi\in\C[z_1,\ldots, z_n]$ is a homogeneous symmetric polynomial of
degree $\ell$ with $n-\kappa<\ell\leq n$, then it follows by Theorem~ \ref{thm:vanishing}
and Lemma ~\ref{lemma:exact-curvature} that $\Phi(D)$ vanishes outside of $\Sigma$,
and in particular, it has compact support in $U$.
Since $Z'$ is a deformation retract of $U$, the homology groups of $U$ and $Z'$
are naturally isomorphic. Composing this isomorphism with the Poincaré duality
$H^{2\ell}_c(U,\C) \simeq H_{2n-2\ell}(U,\C)$ yields an isomorphism
$H^{2\ell}_c(U,\C) \simeq H_{2n-2\ell}(Z',\C)$.
Now $\res^\Phi(\F;Z') \in H_{2n-2\ell}(Z',\C)$ is defined as the class of $\Phi(D)$ in
$H^{2 \ell}_c(U,\C)$ under this isomorphism.
It is proved in \cite[Sections~5,6,7]{baum-bott} that the class of $\Phi(D)$ is independent
of the choices of $U,(E,\varphi),D$, and only depends on the local behaviour of
$\F$ around $Z'$.

\section{Residue currents}\label{stream}

We say that a function $\chi:\R_{\geq 0}\to \R_{\geq 0}$ is a \emph{smooth approximant
  of the characteristic function} $\chi_{[1,\infty)}$ of the interval
  $[1,\infty)$ and write
  \[
    \chi \sim \chi_{[1,\infty)}
  \]
 if $\chi$ is smooth, increasing and $\chi(t) \equiv 0$ for $t \ll 1$ and
 $\chi(t) \equiv 1$ for $t \gg 1$.

\subsection{Pseudomeromorphic currents}\label{hemma}
Let $f$ be a (generically nonvanishing) holomorphic function on a
(connected)
complex manifold $M$.
Herrera and Lieberman, \cite{HL}, proved that the \emph{principal value}
\begin{equation*}
\lim_{\epsilon\to 0}\int_{|f|^2>\epsilon}\frac{\xi}{f}
\end{equation*}
exists for test forms $\xi$ and defines a current, that we with a slight abuse of notation denote by $1/f$.
It follows that $\dbar(1/f)$ is a current with support on the zero set
$Z(f)$ of $f$; such a current is called a \emph{residue current}.
Assume that $\chi\sim\chi_{[1,\infty)}$ and that
 $s$ is a generically
nonvanishing holomorphic section of a Hermitian vector bundle such that
$Z(f)\subseteq \{s = 0\}$.
Then
\begin{equation*}
  \frac{1}{f}=\lim_{\epsilon\to 0} \frac{\chi(|s|^2/\epsilon)}f ~~~~~ \text{
    and } ~~~~~
  \dbar\left (\frac{1}{f}\right )=\lim_{\epsilon\to 0} \frac{\dbar\chi(|s|^2/\epsilon)}f,
  \end{equation*}
  see, e.g., \cite{andersson-wulcan:asm}.
  In particular, the limits are independent of $\chi$ and $s$. Note
  that $\chi(|s|^2/\epsilon)$ vanishes identically in a neighborhood
  of $\{s=0\}$, so that $\chi(|s|^2/\epsilon)/f$ and $\dbar\chi(|s|^2/\epsilon)/f$ are smooth.
More generally, if $f$ is a generically non-vanishing holomorphic section of a line bundle $L\to M$ and $\omega$ is an $L$-valued smooth form, then the current $\omega \slash f$ is well-defined. Such currents are called \textit{semi-meromorphic}, cf.\ \cite[Section~4]{andersson-wulcan:asm}.

In \cite{andersson-wulcan:crelle} the sheaf  $\PM_M$ of {\it pseudomeromorphic currents} on
$M$ was introduced in order to obtain a coherent approach to questions about residue and
principal value currents; it consists of direct images under holomorphic mappings of products
of test forms and currents like $1/f$ and $\dbar(1/f)$, and also suitable products of such
currents. See, e.g., \cite[Section~2.1]{andersson-wulcan:asm} for a precise definition.
The sheaf $\PM_M$ is closed under $\partial$ and $\dbar$ and under multiplication by smooth forms.
Pseudomeromorphic currents have a geometric nature, similar to closed positive (or normal) currents.
For instance, the {\it dimension principle} states that if the pseudomeromorphic current
$\mu$ has bidegree $(*,p)$ and support on a variety of codimension strictly larger than
$p$, then $\mu$ vanishes.

The sheaf $\PM_M$ admits natural restrictions to constructible subsets
of $M$.
In particular, if $W$ is a subvariety of the open subset $U\subseteq M$,
and $s$ is a holomorphic section of a Hermitian vector bundle such that $\{s = 0\} = W$, then the restriction to
$U\setminus W$ of a pseudomeromorphic current $\mu$ on $U$ is
the pseudomeromorphic current on $U$ defined by
\[
    \1_{U\setminus W} \mu := \lim_{\epsilon \to 0} \chi(|s|^2/\epsilon) \mu|_{U},
\]
where $\chi \sim \chi_{[1,\infty)}$.
It follows that
    \[\1_{W} \mu := \mu -
    \1_{U\setminus W} \mu\]
  has support on $W$.
  These definitions are independent of the choice of $s$ and $\chi$.

\subsection{Almost semi-meromorphic currents}

We refer to \cite[Section~4]{andersson-wulcan:asm} for details of the
results mentioned in this section.

We say that a  current $a$ is \textit{almost semi-meromorphic} in $M$, $a\in ASM(M)$,  if there exists a modification $\pi: M' \to M$ and a semi-meromorphic current $\omega \slash f$  on $M'$ such that $a = \pi_* \big( \omega \slash f \big)$. More generally, if $E$ is a vector bundle over $M$, an $E$-valued current $a$ is \textit{almost semi-meromorphic} on $M$ if $ a = \pi_* \big( \omega \slash f \big)$, where $\pi$ is as above, $\omega$ is a smooth form with values in $L \otimes \pi^*E$ and $f$ is a holomorphic section of a line bundle $L\to M'$.

Clearly almost semi-meromorphic currents are pseudomeromorphic. In particular, if $a\in ASM(M)$, then $\partial a$ and $\dbar a$ are pseudomeromorphic currents on $M$.

\begin{lma}[Proposition~4.16 in \cite{andersson-wulcan:asm}]\label{gold}
  Assume that $a\in ASM (M)$ is smooth in $M\setminus W$, where $W$ is
  subvariety of $M$.  $\partial a\in ASM (M)$ and $\1_{M\setminus W}\dbar a\in ASM(M)$.
\end{lma}

Given $a\in ASM(M)$, let $ZSS(a)$ (the \emph{Zariski-singular support}) denote
the smallest Zariski-closed set $V\subset M$ such that $a$ is smooth outside
$V$. The pseudomeromorphic current $r(a):=\1_{ZSS(a)} \dbar a$ is called the
\emph{residue} of $a$.

Almost semi-meromorphic currents have the so-called \emph{standard
  extension property (SEP)} meaning that $\1_Wa=0$ in $U$ for each
subvariety $W\subset U$ of positive codimension, where $U$ is any open
set in $M$.
In particular, if $a\in ASM(M)$, $\chi \sim \chi_{[1,\infty)}$, and $s$ is any generically non-vanishing holomorphic section of a Hermitian
vector bundle over $M$,
then
\begin{equation}\label{tullen}
    \lim_{\epsilon\to 0} \chi(|s|^2/\epsilon) a= a.
\end{equation}
It follows in view of Lemma ~\ref{gold} that, if $\{s=0\}\supset ZSS(a)$, then
\begin{equation}
 r(a) = \lim_{\epsilon\to 0}  \dbar \chi(|s|^2/\epsilon) \wedge a = \lim_{\epsilon\to 0}  d \chi(|s|^2/\epsilon) \wedge a.
\end{equation}

\begin{lma} \label{lma:regProduct}
Let $\chi$ be a cutoff function, $Z \subseteq M$ an analytic subset of positive codimension,
and let $s_1,\ldots,s_k$ be sections of vector bundles $F_1,\ldots,F_k$ such that
$s_j \not\equiv 0$ and $\{ s_k = 0 \} \supseteq Z$, and let $\beta$ be an almost semi-meromorphic
current on $M$. Then
\begin{equation*}
    \lim_{\epsilon \to 0} \chi(|s_1|^2/\epsilon) \cdots \chi(|s_k|^2/\epsilon) \beta = \beta.
\end{equation*}
\end{lma}

\begin{proof}
We prove more generally that
\begin{equation} \label{eq:betaRegularization}
    \chi(|s_1|^2/\epsilon_1) \cdots \chi(|s_k|^2/\epsilon_k) \beta
\end{equation}
is continuous as a function of $(\epsilon_1,\ldots,\epsilon_k) \in [0,\infty)^k$, and
the value at $0$ equals $\beta$.
The desired equality then holds since $\beta$ has the SEP, and thus the
iterated limit of \eqref{eq:betaRegularization} when $\epsilon_k \to 0$
one at a time equals $\beta$.

Since $\beta$ is a locally finite sum of push-forwards under modifications of currents
of the form $\pi_*(\beta_0/f)$, where $\beta_0$ is a smooth form with compact support
and $f$ is a holomorphic function, we may in fact assume that $\beta = \beta_0/f$.
The result then follows by \cite[Theorem~1.1]{BS} and the paragraph after the proof
of Theorem~1.1 on \cite[p. 52]{BS}.
\end{proof}

\begin{prop} \label{prop:globalAsmRegularization}
Let $(U_\alpha)$ be a cover of $M$, let $(\psi_\alpha)$ be a partition of unity subordinate to
$(U_\alpha)$, and let $\beta$ be an almost semi-meromorphic current on $M$,
which is smooth outside of the analytic subset $Z$.
Let for each $\alpha$, $s_\alpha\not\equiv 0$ be a section of some Hermitian vector
bundle over $U_\alpha$ such that $\{ s_\alpha = 0 \} \supseteq Z$,
let $\chi : \R_{\geq 0} \to \R_{\geq 0}$ be a smooth cut-off function,
and define $\chi_\epsilon := \sum \psi_\alpha \chi(|s_\alpha|^2/\epsilon)$.
Then for $k \geq 1$,
\begin{align}
\lim_{\epsilon\to 0} \chi_\epsilon^k \beta = \beta \label{eq:regChi}    \\
\lim_{\epsilon\to 0} d \chi_\epsilon^k \wedge \beta = r(\beta) \label{eq:regDbarChi},
\end{align}
so in particular, both limits are pseudomeromorphic and depend only on $\beta$.
\end{prop}

\begin{proof}
The equality \eqref{eq:regChi} follows immediately from Lemma~\ref{lma:regProduct}.
To prove \eqref{eq:regDbarChi}, we have that
\begin{equation*}
\lim_{\epsilon\to 0}  d\chi_\epsilon^k \wedge \beta = \lim_{\epsilon\to 0} d(\chi_\epsilon^k \beta) - \chi_\epsilon^k (\partial \beta + \1_{M\setminus Z} \dbar\beta) = d\beta - \partial\beta - \1_{M\setminus Z} \beta = r(\beta),
\end{equation*}
where the first equality holds since $\1_{M\setminus Z} \dbar\beta = \dbar\beta$
on $\supp \chi_\epsilon^k$, and the second equality holds by \eqref{eq:regChi} since
since $\partial \beta$ and $\1_{M\setminus Z} \dbar \beta$ are almost semi-meromorphic by
Lemma~\ref{gold}.
\end{proof}

If $a_1, a_2 \in ASM(M)$, then $a_1+a_2\in ASM(M)$, and moreover there
is a well-defined product $a_1\wedge a_2\in ASM(M)$, so that $ASM (M)$
is an algebra over smooth forms, see \cite[Section~4.1]{andersson-wulcan:asm}. Note that
if $\chi \sim \chi_{[1,\infty)}$ and $s$ is a generically nonvanishing
holomorphic section of a Hermitian vector bundle such that $\{s=0\}$
contains the Zariski-singular supports of $a_1$ and $a_2$,  then
$a_1\wedge a_2$ is the limit of the smooth form $\chi(|s|^2/\epsilon) a_1\wedge a_2.$

\section{Residue currents as localization of characteristic classes} \label{sect:recallKLW}

We briefly recall the relevant parts of the constructions from \cite{LW:chern-currents-sheaves}
and \cite{KLW} of residue currents representing characteristic classes of coherent sheaves.

Assume that $\G$ is a coherent analytic sheaf on $M$, and that $\G$ admits a
locally free resolution
\begin{equation}
    0 \to E_N \stackrel{\varphi_N}{\rightarrow} \cdots \stackrel{\varphi_1}{\rightarrow} E_0 \to \G \to 0,
\end{equation}
and assume that $E$ is equipped with Hermitian metrics and a connection $D$.
The construction of the residue currents in both these articles is done through
a procedure as follows:
If $Z$ denotes the analytic subset where $\G$ is not a vector bundle,
then under suitable hypothesis, a connection $\tilde{D}$ on $E|_{M\setminus Z}$ is
constructed,and $\Phi \in \C[z_1,\ldots,z_n]$ is a homogeneous symmetric polynomial
of appropriate degree, then $\Phi(\tilde{D}) = 0$.

Let $s$ be a holomorphic section of some holomorphic vector bundle over $M$ such
that $\{ s = 0 \} = Z$. For example, we may take $s = \bigwedge^\rho \varphi_1$,
where $\rho$ is the (generic) rank of $\varphi_1$.
Let $\chi \sim \chi_{[1,\infty)}$ and define $\chi_\epsilon = \chi(|s|^2/\epsilon)$.
Note that $\chi_\epsilon \equiv 0$ in a neighborhood of $Z$,
and is such that if  that the set $\Sigma_\epsilon$ where $\chi_\epsilon$ is not identically
$1$ shrinks to $Z$ when $\epsilon$ tends to $0$.
One may then define a new family of connections
\begin{equation} \label{eq:DtildeRegularization}
    \widehat{D}^\epsilon = (1-\chi_\epsilon) \tilde{D} + \chi_\epsilon D.
\end{equation}
Then, $\Phi(\widehat{D}^\epsilon)$ will be a family of characteristic forms,
which since $\widehat{D}^\epsilon = \tilde{D}$ outside of $\Sigma_\epsilon$,
this family will satisfy that the support of $\Phi(\widehat{D}^\epsilon)$
shrinks to $Z$ as $\epsilon$ tends to $0$.
The main point of the construction is then to choose the connections $\tilde{D}$
appropriately such that the limit
$R^\Phi = \lim_{\epsilon\to 0} \Phi(\widehat{D}^\epsilon)$ exists as a current,
which will then have support on $Z$, and represent $\Phi(\G)$.
We describe below how this is done in the two cases.

In both cases, the connections are constructed such that if one for $k=0,\dots,N$,
writes on $M \setminus Z$
\begin{equation} \label{eq:tildeDk}
    \tilde{D}_k = D_k + a_k,
\end{equation}
where $a_k$ is a smooth $\End(E_k)$-valued $1$-form on $M\setminus Z$, then $a_k$
admits an extension as an almost semi-meromorphic current to $Z$.
It will then follow from the theory of almost semi-meromorphic currents
that the limit $\Phi(\widehat{D}^\epsilon)$ exists.
Note that when $\tilde{D}_k$ is of this form, the regularization \eqref{eq:DtildeRegularization}
may be written as
\begin{equation} \label{eq:DtildeRegularization2}
    \widehat{D}^\epsilon = \chi_\epsilon a + D.
\end{equation}
If $D=(D_N,\dots,D_0)$, we say that $D$ is of type $(1,0)$, or that $D$ is a $(1,0)$-connection,
if $D_k$ is of type $(1,0)$ for each $k$.
\begin{remark} \label{axel}
Note that if $D$ and $\tilde{D}$ are of type $(1,0)$, then so is $\tilde{D}^\epsilon$.
\end{remark}

In both cases, the morphisms $a_k$ are constructed in terms of the
following morphisms. For $k=1,\ldots,N$, we let $\sigma_k: E_{k-1}\to E_k$
be the \emph{minimal inverse} of $\varphi_k$. These are smooth vector bundle
morphisms defined outside the analytic set $Z_k\subset Z$ where $\varphi_k$
does not have optimal rank and are determined by the following properties:
\begin{equation*}
\varphi_k \sigma_k \varphi_k = \varphi_k, \quad \im \sigma_k \perp \im \varphi_{k+1} \quad \text{and} \quad  \sigma_{k+1} \sigma_k = 0.
\end{equation*}
The morphisms $\sigma_k$ may be continued as almost semi-meromorphic sections
of $\End (E)$, see, e.g., the proof of Lemma~2.1 in \cite{LW:chern-currents-sheaves}.

\begin{lma} \label{lma:DtildeLW}
Let $\G,Z,(E,\varphi),D$ be as above, and assume that $Z=\supp \G$ has codimension $\geq 1$.
Define a connection $\tilde{D}$ on $E$ through \eqref{eq:tildeDk} with
$a_k = -\sigma_k D\varphi_k$ for $k=1,\dots,N$.
Then $\tilde{D}$ is a connection on $M \setminus Z$ which is compatible with $(E,\varphi)$
and $a_k$ admits an extension as an almost semi-meromorphic current to $M$.
If $D$ is of type $(1,0)$, then so is $\tilde{D}$.
\end{lma}

\begin{proof}
This definition of $\tilde{D}_k$ coincides with the definition in
\cite[Remark~4.5]{LW:chern-currents-sheaves}. It also follows by that remark
that $\tilde{D}$ is compatible with $(E,\varphi)$ since $(E,\varphi)$ is
pointwise exact on $M \setminus Z$.
Since $\sigma_k$ extends as an almost semi-meromorphic current to $M$,
and almost semi-meromorphic currents are closed under multiplication with smooth
forms, it follows that $a_k$ extends as an almost semi-meromorphic current to all of $M$.

If $D_k$ and $D_{k-1}$ are of type $(1,0)$, then since $\varphi_k$ is holomorphic,
$D\varphi_k$ will be of bidegree $(1,0)$, and hence $\tilde{D}_k = D_k + a_k$ will be
of type $(1,0)$ as well.
\end{proof}

We now consider the case of foliations from \cite{KLW}, i.e., we assume that $\G = N\F$,
where $\F$ is a holomorphic foliation on $M$. We furthermore assume that $N\F$ admits a
resolution of the form
\begin{equation}
    0 \to E_N \stackrel{\varphi_N}{\rightarrow} \cdots \stackrel{\varphi_1}{\rightarrow} TM \to N\F \to 0,
\end{equation}
i.e., at level $0$, the bundle equals $TM$, and that the connection $D_0$ on $TM$
is of type $(1,0)$ and torsion free.
In \cite{KLW}, we then defined a certain smooth morphism
$\mathcal{D}\varphi_1 : E_0 \otimes TM \to E_1$, which is obtained from $D\varphi_1$
by the identification of $1$-forms with $\Hom(TM,\Ok)$, see
\cite[equation (2.11)]{KLW} for a precise definition.
Using this morphism, one may produce a new morphism
\begin{equation} \label{eq:bDef}
    b = \mathcal{D}\varphi_1 \sigma_1 \left( dz\cdot \frac{\partial}{\partial z}\right),
\end{equation}
which takes values in $\End(E_0)$ and is smooth on $M\ Z$, see \cite[equation (5.7)]{KLW},
and where $dz \cdot \frac{\partial}{\partial z}$ denotes the $1$-form valued vector
in $E_0 = TM$ which invariantly may be described as the vector induced by $\Id_{TM}$
through identifying elements of $\End(TM)$ with elements of $\Hom(TM,\Ok) \otimes TM$,
i.e., $1$-form valued sections of $TM$.

\begin{lma} \label{lma:DtildeKLW}
Let $\G,Z,(E,\varphi),D$ be as above, and assume that $\G = N\F$, where $\F$ is a
holomorphic foliation on $M$.
Define a connection $\tilde{D}$ on $E$ through \eqref{eq:tildeDk} with
\[
a_0 = b(\Id_{E_0}-\varphi_1 \sigma_1)-D\varphi_1\sigma_1
\quad \text{ and } \quad
a_k = -D\varphi_{k+1}\sigma_{k+1} \text{\quad for \quad} k=1,\dots,N,
\]
where $b$ is defined by \cite[equation (5.7)]{KLW}.
Then $\tilde{D}$ is a connection on $M \setminus Z$ which is compatible with $(E,\varphi)$,
$a_k$ admits an extension as an almost semi-meromorphic current to $M$,
and if $\tilde{D}_{-1}$ denotes the connection on $E_{-1} := \coker \varphi_1|_{M\setminus Z} \cong N\F|_{M \setminus Z}$ induced by $\tilde{D}_0$, then $\tilde{D}_{-1}$ is a basic connection on $E_{-1}$.
Furthermore, the form $b$ is locally defined in terms of $(E,\varphi)$ and its Hermitian metrics
and connection $D$.
If $D_1,\dots,D_N$ are of type $(1,0)$, then so is $\tilde{D}$.
\end{lma}

\begin{proof}
It follows by \cite[equation (5.7), (5.16), (5.17) and (5.18)]{KLW} that $\tilde{D}$ defined
here coincides with the connections $\tilde{D}$ in \cite{KLW} (while we remark that
the smooth reference connection on $E_0$ that we here denote by $D_0$ is denoted
$D^{TM}$ in \cite{KLW}, and the $D_0$ that appears in \cite[Section 5.1]{KLW} has a different meaning).
It follows similarly as above that $a_k$ is smooth on $M \setminus Z$, and
by additionally using that almost semi-meromorphic currents form an algebra,
it follows that $a_k$  admits an extension as an almost semi-meromorphic current on $M$, see
also \cite[Lemma 5.5]{KLW} (where we remark that the form $a_0$ defined here is different,
but related to the form $a_0$ in \cite{KLW}).
By \cite[Proposition~5.4]{KLW}, $(\tilde{D}_N,\ldots,\tilde{D}_0,D_{basic})$ is compatible
with the augmented complex \eqref{eq:augmentedEcomplex}, so by the discussion in Section~\ref{superduper},$\tilde{D}=(\tilde{D}_N,\dots,\tilde{D}_0)$ is compatible with $(E,\varphi)$,
and $D_{basic}$ is the connection on $E_{-1}$ induced by $\tilde{D}_0$.
By the definition of $b$, it follows that on an open set $U \subseteq M$, it only depends
$(E,\varphi)$ and its Hermitian metrics and connections (at levels $0$ and $1$) restricted to $U$.
\end{proof}

\section{Global Chern forms} \label{sec:globalChern}

In this section we briefly recall the necessary parts of the construction of simplicial
resolutions and characteristic classes due to Green,
\cite{green}.
Further details can be found in \cite{green,TT86,HosI,HosII}.

\subsection{Simplicial resolutions in the sense of Green}
Assume that $\G$ is a coherent analytic sheaf on $M$, and that
$\U=(U_\alpha)_{\alpha\in I}$ is a (locally finite) Stein open cover of
$M$.
Moreover, assume that for some $N\geq 0$ and for each $\alpha\in I$, there is a locally free
resolution $(E^\alpha, \varphi^\alpha)$ of
$\G|_{U_\alpha}$,
\begin{equation} \label{clown}
     0 \to E^{\alpha}_N \xrightarrow[]{\varphi^\alpha_N}
     \cdots\xrightarrow[]{\varphi^\alpha_1}
     E^{\alpha}_0\xrightarrow[]{} \G|_{U_\alpha} \to 0,
   \end{equation}
   of $\Ok$-modules over $U_\alpha$.
By Hilbert's syzygy theorem, locally we can always find locally free
resolutions of length at most $n$, so after possibly
refining $\U$ we may assume that we have resolutions \eqref{clown} for each
vertex $\alpha$ of $N\U$.

Starting from the locally free resolutions $(E^\alpha,
\varphi^\alpha)_{\alpha\in I}$ Green, \cite{green}, see also \cite{TT86},
constructed a ``simplicial resolution'' of $\G$.
The simplicial resolution consists of, for each $\Delta=(\alpha_0, \ldots, \alpha_p) \in N\U$,
a locally free resolution $(E^\Delta, \varphi^\Delta)$
of $\G|_{U_\Delta}$,
\begin{equation} \label{brown}
     0 \to E^{\Delta}_N \xrightarrow[]{\varphi^\Delta_N}
     \cdots\xrightarrow[]{\varphi^\Delta_1}
     E^{\Delta}_0\xrightarrow[]{} \G|_{U_\Delta} \to 0,
\end{equation}
and the different resolutions fit together in a certain way.
In fact, from the $(E^\alpha,\varphi^\alpha)$ one can construct a so-called twisted resolution of $\G$ in
the sense of Toledo and Tong, \cite{TT78}, and given this Green constructed his resolution.
How the resolutions fit together is summarised in \cite[Proposition~1.4]{green}, see also \cite[p.~ 264]{TT86}.
In particular, for each subsimplex $\tau$ of $\Delta=(\alpha_0,\ldots,\alpha_p)$, there is an isomorphism (in the category
of $\Ok_{U_\Delta}$-modules)
\begin{equation}\label{rainfall}
  E^\Delta \cong E^{\tau} \oplus E_{\tau}^{\Delta},
\end{equation}
where $E_{\tau}^{\Delta}$ is an elementary sequence in the
$E^{\alpha_j}, j=0,\ldots, p$.
Furthermore, if $\omega$ is in turn a subsimplex of $\tau$, then there is an isomorphism (in the category
of $\Ok_{U_\Delta}$-modules)
\begin{equation}\label{rainfall2}
  E_\omega^\Delta \cong E_\omega^{\tau} \oplus E_{\tau}^{\Delta}.
\end{equation}
We will say that $(E^\Delta, \varphi^\Delta)$ is a \emph{(holomorphic)
  simplicial
  resolution of $\G$ with respect to $\U$}.

\begin{remark} \label{rem:smoothOrRealAnalytic}
The construction of Green applies also for sheaves of $\E_M$-modules (smooth functions) or
$\mathcal{A}_M$-modules (real analytic functions), to produce smooth or real analytic simplicial
resolutions out of a family of locally free smooth or real analytic resolutions.
By flatness, a holomorphic simplicial resolution is also a real analytic simplicial resolution, and a real analytic simplicial resolution
is also a smooth simplicial resolution.
All the results in this section hold also for such resolutions, except for the results
concerning $(1,0)$-connections.
\end{remark}

\subsection{Connections and characteristic classes}\label{pingpong}

For each $\alpha\in I$, assume that $E^\alpha$ is equipped with a connection $D^\alpha$.

Let $D^\Delta_\alpha$ be the induced connection on the elementary sequence $E^\Delta_\alpha$,
cf., Section~\ref{apple}, and let $\widetilde D^{\Delta, \alpha}$ be the
connection on $E^\Delta$ induced by $D^\alpha\oplus
D^\Delta_\alpha$.
Note that if $\alpha$ is a vertex of $\tau$, which in turn is a subsimplex of $\Delta$,
then it follows by \eqref{rainfall2} that
\begin{equation} \label{eq:DDeltadef}
    \widetilde{D}^{\Delta,\alpha} = \widetilde{D}^{\tau,\alpha}\oplus D^\Delta_\tau,
\end{equation}
where $D^\Delta_\tau$ is the connection induced by the elementary sequence $E^\Delta_\tau$.

Given $\Delta=(\alpha_0, \ldots, \alpha_p)\in N\U$,
let $\pi$ denote the projection $U_{\Delta} \times \Delta_p \to
U_{\Delta}$, where $\Delta_p$ is the standard simplex \eqref{snowfall},
and let
\begin{equation}\label{redpen}
D^{\Delta} =
\sum_{j=0}^p t_j \pi^* \widetilde D^{\Delta, \alpha_j}.
\end{equation}
Then $D^{\Delta}$ is a connection on $\pi^* E^\Delta$.

Let $\Phi\in \C[z_1, \ldots, z_n]$ be a homogeneous symmetric
polynomial of degree $\ell$.
On $U_\Delta\times \Delta_p$ we have the characteristic form $\Phi(D^\Delta)$ defined
as in \eqref{schmetterling}.
We can associate with $D = (D^\alpha)_{\alpha \in I}$ the \v Cech-de Rham cochain
$\check \Phi(D) \in \check C^{\bullet}(\U, \E^\bullet)$, defined by
\begin{equation} \label{rectangle}
 \check\Phi(D)_{\Delta} = \pi_* \Phi(D^\Delta) = \int_{\Delta_p} \Phi(D^\Delta).
\end{equation}
Note that $\check \Phi(D)_\Delta\in \E^{2\ell-p}(U_\Delta)$, so
that
$\check \Phi(D)$ has total (\v Cech-de Rham) degree $2\ell$.

\begin{prop} \label{prop:Green-CdR}
Assume that $\G$, $\U$, $D$ and $\Phi$ are as above.
Then $\nabla \check \Phi(D) = 0$.
\end{prop}

When $\Phi=e_\ell$, the above proposition is \cite[Lemma 2.2]{green}.

\begin{proof}
Note that \[
    \Phi(D^\Delta)_\Delta =\Phi(\tilde{D}^{\Delta,\alpha_0},\dots,\tilde{D}^{\Delta,\alpha_p}),
\]
where the right-hand side is defined by \eqref{eq:PhiD0Dp}.
It then follows by \eqref{eq:suwa-interpolation}, Lemma~\ref{lma:phiElementarySequence}, \eqref{rainfall}
and \eqref{rainfall2} that $\nabla \check \Phi(D) = 0$.
\end{proof}

\begin{remark}
One could alternatively have defined first \v Cech-de Rham cochains $\Phi(D)$
just for $\Phi=\elem_\ell$ being the elementary symmetric polynomials, and then defined
$\check{\Phi}(D)$ in terms of $\check{\elem}_\ell(D)$ through a formula corresponding to \eqref{schmetterling}.
This would yield a different definition of $\check{\Phi}(D)$, which would be cohomologous
to our definition by an explicit cochain, cf., \cite[Lemma~1.5 and Proposition~1.6]{SuwaCC}.
The results in this article would also hold for this alternative definition, but the one we have chosen seems
to be the most natural and convenient for our results.
\end{remark}

We say that a form on $U_\Delta\times \Delta_p$ has degree $(r,s,q)$
if it (in local coordinates) can be written as a sum of forms of the
form
\begin{equation*}
f dz_{i_1}\wedge \dots\wedge dz_{i_r} \wedge d\bar{z}_{j_1}\wedge\dots\wedge d\bar{z}_{j_s} \wedge dt_{k_1}\wedge \wedge dt_{k_q},
\end{equation*}
where $f$ is a smooth function.
We say that a connection on $\pi^* E^\Delta$ is a $(1,0)$-connection if its connection matrices
in frames induced by local holomorphic frames of $E^\Delta$ are of degree $(1,0,0)$. Note that when $p=0$, it coincides with the usual notion of $(1,0)$-connections.

Note that if the $D^{\alpha_j}$ are $(1,0)$-connections then so is
$D^{\Delta}$. Moreover, in view of Section ~\ref{apple}
note that if $D^{\alpha_j}$ is
compatible with $(E^{\alpha_j}, \varphi^{\alpha_j})$ for $j=0,\ldots, p$, then $D^{\Delta}$ is compatible with $(\pi^*E^\Delta,
\pi^*\varphi^\Delta)$.

\begin{lma} \label{lma:checkPhiBidegree}
Assume that $\G$, $\U$, $D$ and $\Phi$ are as above. Assume furthermore that for each $\alpha \in I$,
$D^\alpha$ is a $(1,0)$-connection. Then the elements of $\check{\Phi}(D)$ are sums of forms of bidegree $(\ell+r, s)$,
where $r,s\geq 0$  and $r+s+p=\ell$.
\end{lma}

\begin{proof}
Let us fix $\Delta=(\alpha_0,\ldots, \alpha_p)\in N\U$.
Recall
that since the $D^{\alpha_j}$ are $(1,0)$-connections, then so are the
connections $\widetilde D^{\Delta, \alpha_j}$.
Since the statement is local we may
work in a
local trivialization.
Let $\theta_k^{\alpha_j}$ be the connection matrix of $\widetilde
D^{\Delta, \alpha_j}_k$; with a slight abuse of notation we will let
it denote also the connection matrix of $\pi^* \widetilde
D^{\Delta, \alpha_j}_k$. Then the connection matrix of $D^\Delta_k$
equals
\[
  \theta^\Delta_k= \sum_{j=0}^p t_j  \theta_k^{\alpha_j},
\]
cf.\ \eqref{redpen}.
Now the curvature matrix of $D^\Delta_k$ equals
\[
\Theta (D^\Delta_k) =
\theta_k^\Delta\wedge \theta_k^\Delta +
\sum_{j=0}^p ( dt_j \wedge \theta_k^{\alpha_j}+t_j \partial
\theta_k^{\alpha_j}+t_j\dbar \theta^{\alpha_j}_k).
\]
Since $\theta_k^\Delta$ is a $(1,0)$-form it
follows that $\Theta (D^\Delta_k)$ is a sum of
terms of degree $(2,0,0)$, $(1,0,1)$, and $(1,1,0)$. Since
$\elem_i(D^\Delta)$ is a polynomial of degree $j$ in the
entries of $\Theta (D^\Delta_k)$, $k=0,\ldots, N$
it follows that it is a sum of forms of
degree $(j+r, s, q), r+s+q=j$ (since each factor has at least degree
$1$ in $dz_j$). Since $\Phi$ is homogeneous of degree $\ell$,
$\Phi(D^\Delta)$ is a sum of forms of degree $(\ell+r, s, q)$,
$r+s+q=\ell$, where $r,s,q \geq 0$. It follows that  $\check \Phi(D)_\Delta=\int_{\Delta_p} \Phi
(D^\Delta)$ is a sum of forms of bidegree $(\ell+r, s)$, where $r,s\geq
0$ and $r+s+p=\ell$, since integration of a form of degree $(*,*,q)$ over $\Delta_p$ vanishes unless $q=p$.
\end{proof}

The next result shows that $[\check{\phi}(D)]$ only depends on $\mathcal{G}$ in a suitable sense.
To formulate the result, we consider two Stein open covers $\U_1=(U_\alpha)_{\alpha \in I_1}$ and $\U_2=(U_\beta)_{\beta \in I_2}$ of $M$.
We define a new cover $$\U_{12} =\U_1\cap \U_2 = (U_\alpha \cap U_\beta)_{(\alpha, \beta)\in I_1\times I_2}.$$
Let $r_1:I_1\times I_2\to I_1$ be the projection $(\alpha,\beta)\mapsto \alpha$ and define $r_2$ analogously; then for $j=1,2$,
$\U_{12}$ refines $\U_j$ by the refinement map $r_j$.

\begin{prop} \label{prop:Green-CdR-indep}
Assume that $\G$, $\U_1,\U_2,\U_{12},r_1,r_2$ and $\Phi$ are as above. Assume there is some $N\geq 0$ such that for $j=1,2$, and for each $\alpha \in I^j$, $\G|_{U_\alpha}$ admits a locally free resolution
$(E^{\alpha},\varphi^\alpha)$
\begin{equation}
0 \to E_N^\alpha \stackrel{\varphi^\alpha_N}{\longrightarrow} E_{N-1}^\alpha \stackrel{\varphi^\alpha_{N-1}}{\longrightarrow} \dots \stackrel{\varphi_2}{\longrightarrow} E_1^\alpha \stackrel{\varphi^\alpha_1}{\longrightarrow} E_0^\alpha
\stackrel{\varphi^\alpha_0}{\longrightarrow} \mathcal{G}|_{U_\alpha} \to 0,
\end{equation}
and assume that $E^\alpha$ is equipped with a connection $D^\alpha$.
For $j=1,2$, let $D^j=(D^\alpha)_{\alpha \in I^j}$.

Then there is a cochain $\check{\eta}^\Phi \in \oplus_{p+q=2\ell-1} \check{C}^p(\U_{12},\E^q)$ such that
\begin{equation} \label{eq:nablaPotential1}
    \nabla \check{\eta}^\Phi = r_2 \check{\Phi}(D^2) - r_1 \check{\Phi}(D^1).
\end{equation}
The elements of $\check{\eta}^\Phi$ are polynomials in the entries of
$\theta^\alpha$ and $d\theta^\alpha$, where $\alpha \in I_1 \coprod I_2$,
and $\theta^\alpha$ is the connection matrix of $D^\alpha$ in some local frame.

If each $D^\alpha$ is a $(1,0)$-connection, then
\begin{equation} \label{eq:etaBidegree}
  \check \eta^\Phi \in \bigoplus_{r\geq 0, ~ r+s+p=\ell-1}
  \check C^p({\U_{12}}, \E^{\ell+r, s}).
\end{equation}
\end{prop}

\begin{proof}
We let $$\V= \U_{1} \coprod \U_{2}=(U_\alpha)_{\alpha\in I_1\coprod I_{2}}.$$
For $j=1,2$ let $\rho_j:I_j \to I_1\coprod I_2$ be the inclusion $\alpha\mapsto
\alpha$; then $\U_j$ refines $\V$ by the refinement map $\rho_j$.
Let $\varrho_j=\rho_j\circ r_j : I_1\times I_2 \to I_1\coprod I_2$ be
the corresponding refinement maps of the refinement of $\V$ in $\U$, and let
$h=h^p:\check C^p(\V, \E^\bullet)\to \check C^{p-1}(\U_{12}, \E^\bullet)$
be the associated homotopy \eqref{semla}.

From $(E^\alpha, \varphi^\alpha)_{\alpha\in I_1 \coprod I_2}$ one can construct a
simplicial resolution $(E^\Delta, \varphi^\Delta)_{\Delta \in N\V}$ of
$\mathcal{G}$ with respect to $\V$ such that on ``pure'' subsets $U_\Delta,
\Delta \in N\U_j$, it coincides with the original simplicial resolutions
$(E^\Delta, \varphi^\Delta)_{\Delta\in N\U_j}$ with respect to
$\U_j$, see, e.g., the proof of Theorem ~2.4 in \cite{green}.

As above we equip $(E^\Delta, \varphi^\Delta)_{\Delta\in N\V}$ with the connections
$(D^{\Delta}_N, \ldots, D^\Delta_0)$ constructed as above,
and let $D=(D^\alpha)_{\alpha \in I_1 \coprod I_2}$.
We denote the \v Cech-de Rham cocycle $\check \Phi( D)$
by $\check \Phi_\V=\check \Phi_\V(D)$ to emphasize the dependence on $\V$.
Then, since $(E^\Delta, \varphi^\Delta)_{\Delta\in N\V}$ coincides with
$(E^\Delta, \varphi^\Delta)_{\Delta\in N\U_j}$ for $\Delta \in
N\U_j$, it follows that for $\Delta \in N\U_j$, $\check
\Phi_\V(D)_\Delta= \check \Phi_{\U_j}(D^j)_\Delta$, where $\check
  \Phi_{\U_j} =\check \Phi_{\U_j}(D)$ denotes the \v Cech-de Rham cocycle associated with the
  simplicial resolutions with respect to $\U_j$.
In other words
\begin{equation}\label{mardigras}
    \rho_j \check \Phi _\V = \check \Phi _{\U_j}.
\end{equation}

Let
\[
    \check \eta^\Phi = h \check \Phi_\V(D) \in \bigoplus_{p+q=2\ell-1} \check
C^{p}(\U_{12}, \E^q).
\]
Then
\begin{equation}\label{kremla}
    \nabla \check \eta^\Phi = (\nabla h+ h\nabla) \check \Phi_\V =
    \varrho_2 \check \Phi_\V - \varrho_1 \check \Phi_\V =
    r_2 \check \Phi_{\U_2} - r_1 \check \Phi_{\U_1},
\end{equation}
where we have used that $\check \Phi_\V$ is a cocycle for the
first equality, \eqref{almond} for the second equality, and
\eqref{mardigras} for the last equality.

By the definition of $\Phi_\V$ and $h$, \eqref{semla}, it follows that the elements of
$\check{\eta}^\Phi$ are polynomials in the entries of $\theta^\alpha$ and $d\theta^\alpha$,
where $\alpha \in I_1 \coprod I_2$, and $\theta^\alpha$ is the connection matrix
of $D^\alpha$ in some local frame.

Assume now that for each $\alpha\in I_1\coprod I_2$, $D^\alpha$ is a $(1,0)$-connection.
By Lemma~\ref{lma:checkPhiBidegree},
\[
    \check \Phi_\V \in \bigoplus_{r\geq 0, ~ r+s+p=\ell}
    \check C^p(\V, \E^{\ell+r, s}).
\]
Moreover, note that $h$ maps $\check C^p(\V, \E^{\ell+r, s})$ to
$\check C^{p-1}({\U_{12}}, \E^{\ell+r, s})$, and thus
\[
  \check \eta^\Phi \in \bigoplus_{r\geq 0, ~ r+s+p=\ell-1}
  \check C^p({\U_{12}}, \E^{\ell+r, s}).
\]
It then follows that \eqref{eq:etaBidegree} holds.
\end{proof}

If $\U_1 = \U_2 = \U$ in the above proposition, then one may alternatively find $\check{\eta}^\Phi \in \oplus_{p+q=2\ell} \check{C}^p(\U,\E^q)$ such that
\begin{equation}
    \nabla \check{\eta}^\Phi = \check{\Phi}(D^2) - \check{\Phi}(D^1).
\end{equation}
This follows from making minor adjustments to the above proof, i.e.,
letting $\U_{12} = \U$, and letting $r_1$ and $r_2$ be the identity maps.

\begin{remark}\label{biathlon}
Assume that $\Delta\in N\U$ is a vertex, i.e.,
  $\Delta=(\alpha)$.
Then note that $\Delta_p=\{1\}$. In this case we identify
$U_\Delta\times \Delta_p$ with $U_\Delta$ and then $\pi$ is just the
identity. Then $D^\Delta = D^\alpha$ and hence $\check \Phi(D)_\Delta
= \Phi(D^\alpha)$.
  \end{remark}

  \subsection{Characteristic forms}\label{ronaldo}
  Assume that $(\psi_\alpha)_{\alpha\in I}$ is a partition of unity
  subordinate to $\U$ and let $\Psi$ be the corresponding morphism \eqref{cream}. Then we define the
  characteristic form $\Phi(D)$ as
\begin{equation}\label{cardamom}
    \Phi(D) = \Psi(\check{\Phi}(D)).
\end{equation}
Since $\check\Phi(D)$ is a \v Cech-de Rham cocycle of total degree
$2\ell$ and $\Psi$ is an isomorphism on cohomology, it follows that
$\Phi(D)$ is a closed form of degree $2\ell$.
The de Rham cohomology class that $\Phi(D)$ defines indeed only depends on $\mathcal{G}$, cf.,
the proof of Lemma~\ref{hamburger} below.

\begin{ex}\label{board}
 Assume that
\begin{equation} \label{down}
     0 \to E_N \xrightarrow[]{\varphi_N}  \cdots\xrightarrow[]{\varphi_1}  E_0\xrightarrow[]{} \G \to 0
   \end{equation}
   is a locally free resolution of $\G$ in $M$, and assume that
   $\U=(U_\alpha)_{\alpha\in I}$ is a Stein cover of
$M$. Then $(E^\Delta, \varphi^\Delta)_{\Delta\in N\U}$, where
$(E^\Delta, \varphi^\Delta)=(E, \varphi)$, is a simplicial resolution
of $\G$ with respect to $\U$. In this case each $E^\Delta_\alpha$ is the zero complex, cf.\ \eqref{rainfall}.

Assume that $(E, \varphi)$ is equipped with a connection $D$, and
for $\alpha\in I$, equip $(E^\alpha, \varphi^\alpha)$ with the
connection $D^\alpha:=D$.
Take $\Delta=(\alpha_0, \ldots, \alpha_p)$, $p>0$,
and let $D^\Delta$ be the connection on $\pi^*E^\Delta$ as defined in
Section ~\ref{pingpong}. Using the notation in that section, note that
$\widetilde D^{\Delta, \alpha}=D^\alpha=D$ and thus, in view of
\eqref{redpen} $D^\Delta = \pi^*D$.
If follows that $\Phi(D^\Delta)$ is independent of $t$ and thus
$\check \Phi (D)_\Delta =0$. If $\Delta=(\alpha)$, then $D^\Delta=D$
and thus $\check \Phi(D)_{(\alpha)}=\Phi(D)|_{U_\alpha}$, where
$\Phi(D)$ within this example denotes
the characteristic form \eqref{schmetterling}, cf.\
Remark ~\ref{biathlon} defined from \eqref{down}.

Assume that $(\psi_\alpha)_{\alpha\in I}$ is a partition of unity
subordinate to $\U$, and let $\Psi$ be the corresponding morphism
\eqref{cream}. Then the associated characteristic form
\eqref{cardamom}, which we within this example denote by
$\Phi_{Green}(D)$,
is given by
\[
  \Phi_{Green}(D) = \Psi (\check \Phi(D) ) =
  \sum_{\alpha\in I} \psi_\alpha \check \Phi(D)_{(\alpha)} =
  \sum_{\alpha\in I} \psi_\alpha \Phi(D) =
  \Phi(D).
\]
  \end{ex}

\begin{lma}\label{candy}
  Assume that for each $\alpha\in I$, $D^\alpha$ is a
  $(1,0)$-connection. Then $\Phi(D)$ is a
  sum of forms of bidegree $(\ell+j, \ell-j)$, $0\leq j \leq \ell$.
\end{lma}

\begin{proof}
 Note that $\Psi$ maps an element in $\check C^p(\U, \E^{q_1, q_2})$
  to a sum of forms of bidegree $(q_1+p_1, q_2+p_2)$, where
  $p_1,p_2\geq 0$ and $p_1+p_2=p$.
  It then follows by Lemma~\ref{lma:checkPhiBidegree} that $\Phi(D)=\Psi(\check \Phi(D))$
  is a sum of forms of bidegree $(\ell+r+p_1, s+p_2)$, where $r, p_1, s, p_2\geq 0$ and
  $\ell+r+p_1+s+p_2=2\ell$, i.e., $\Phi(D)$ is of the
  desired form.
\end{proof}

\section{Characteristic forms of coherent sheaves}

Assume that $\G$ is a coherent sheaf on $M$, and that $Z$ denotes the set where $\G$ is not
a vector bundle, and that either $\supp \G$ has codimension $\geq 1$ and then let $\kappa = n$,
or that $\G = N\F$, where $\F$ is a holomorphic foliation of rank $\kappa$ on $M$.
Assume also that $\U=(U_\alpha)_{\alpha\in I}$ is a Stein open cover and $N \in \N$
is such that in each $U_\alpha$ there is a locally free resolution
$(E^\alpha,\varphi^\alpha)$ of $\G|_{U_\alpha}$ of length $\leq N$.
Let $(E^\Delta, \varphi^\Delta)_{\Delta\in N\U}$ be the corresponding
simplicial resolution.
Assume also that $\Phi$ is a symmetric polynomial in $\C[z_1,\ldots,z_n]$ of degree $\ell$,
where $n-\kappa < \ell \leq n$.

In this section, all the results, except for the last part of Lemma~\ref{mercredi}
which concerns connections of type $(1,0)$, hold when the resolutions $(E^\alpha,\varphi^\alpha)$
are complexes of smooth or real analytic vector bundles, cf., Remark~\ref{rem:smoothOrRealAnalytic}.

\begin{lma}\label{mercredi}
Assume that $M$, $\G$, $Z$, $\Phi$, $\kappa$, $\U=(U_\alpha)_{\alpha \in I}$,
$(E^\alpha,\varphi^\alpha)$ are as above.
Moreover, assume that there is a closed neighborhood $\Sigma\subset M$ of $Z$,
and for each $\alpha\in I$, assume that $E^\alpha$ is equipped with a connection $D^\alpha$
which is compatible with $(E^\alpha,\varphi^\alpha)$ over $U_\alpha \setminus \Sigma$,
and in case $\G = N\F$, assume that the connection $D_{-1}^\alpha$ on $N\F|_{U_\alpha}$
induced by $D_0^\alpha$ on $E_0$ is a basic connection on $U_\alpha \setminus \Sigma$.
Let $\Phi(D)$ be the corresponding characteristic form
\eqref{rectangle}. Then $\check{\Phi}(D)$ and $\Phi(D)$ have support in $\Sigma$.
\end{lma}

For the proof in the case when $\G = N\F$, we need the following slight generalization
of Baum-Bott's vanishing theorem, Theorem ~\ref{thm:vanishing}.

\begin{lma}\label{bottom}
Assume that $\F$ is a regular foliation of rank $\kappa$ on a complex
manifold $M$ of dimension $n$, that $\Phi$ is a homogeneous
symmetric polynomial $\Phi \in \C[z_1,\ldots,z_n]$ of degree $\ell$
with $n-\kappa < \ell \leq n$,
and that $D^0,\ldots, D^p$ are basic connections on $N\F$.
Let $\Delta_p$ be the simplex \eqref{snowfall}, let $\pi$ denote the
projection $M\times \Delta_p\to M$, and let $D$ be the connection
\[
D=\sum_{j=0}^pt_j \pi^* D^j
\]
on $M\times \Delta_p$; here $\pi^*D^j$ denotes the pullback connection.
Then
$$\Phi(D) = 0 \quad \text{on} \quad M \times \Delta_p.$$
\end{lma}
\begin{proof}
This follows essentially by the arguments in the proof of
\cite[Proposition~3.27]{baum-bott}. As in that proof, choose local
coordinates $z_1,\ldots, z_n$ on $M$ such that $T\F$ is generated by
$\partial /\partial z_1, \ldots, \partial/\partial z_\kappa$. Then for
each $j=0,\ldots, p$, by
the arguments in that proof each entry of the connection matrix
$\theta_j$ of $D^j$ is in the ideal generated by
$dz_{\kappa+1},\ldots, dz_n$. It follows that each entry the connection matrix
$\theta=\sum t_j \pi^* \theta_j$ of $D$ is in the ideal $\widetilde I$
generated by
$\pi^* dz_{\kappa+1},\ldots, \pi^* dz_n$, and thus each entry in the
curvature matrix $\Theta (D)$ is in $\widetilde I$. As in that proof it
follows that $\Phi(D) = 0$ since
$\ell>n-\kappa$.
\end{proof}

\begin{proof}[Proof of Lemma ~\ref{mercredi}]
Take $\Delta = (\alpha_0,\ldots,\alpha_p)$, $p \geq 0$, and let
$\tilde{D}^\Delta$ be the connection on $\pi^* E^\Delta$ associated to $(\tilde{D}^\alpha)$,
as defined in Section~\ref{pingpong}.
Since $(D^{\Delta}_N, \ldots,  D^{\Delta}_0)$ is
compatible with $(\pi^* E^\Delta, \pi^* \varphi^\Delta)$ in
$(U_\Delta\times \Delta_p)\setminus \pi^{-1} \Sigma$, it follows by
Lemma ~\ref{lemma:exact-curvature} that
\[
\Phi (D^{\Delta}) = \Phi( D^\Delta_{-1})
\]
in $(U_\Delta\times \Delta_p)\setminus \pi^{-1} \Sigma$, where $D^\Delta_{-1}$ is the
connection on $\pi^* \G|_{U_\Delta \setminus \Sigma}$ induced by $D^\Delta_0$.
In case $\supp \G$ has codimension $\geq 1$, then $\G|_{U_\Delta \setminus \Sigma } = 0$,
so $\Phi ( D^\Delta_{-1})=0$.
Otherwise, $\G = N\F$, and then by Lemma ~\ref{bottom} $\Phi ( D^\Delta_{-1} )=0$
in $(U_\Delta\times \Delta_p)\setminus \pi^{-1} \Sigma$.
In any case, $\Phi (D^{\Delta}) =0$ in $(U_\Delta\times \Delta_p)\setminus \pi^{-1} \Sigma$.
It follows that $\check\Phi(D)_\Delta$ has support in $U_\Delta\cap
\Sigma$ and consequently $\Phi(D)$ has support in $\Sigma$ as well, cf.\ \eqref{cream}.
\end{proof}

\begin{lma}\label{hamburger}
Assume that $M, \G$, $Z$, $\Sigma$, $\kappa$ and $\Phi$ are as above.
Assume that there is some $N \in \N$ such that for $j=1,2$,
$\U_j=(U_\alpha)_{\alpha \in  I_j}$ is a Stein open cover
of $M$ such that for each $\alpha\in I_j$, there is a locally free
resolution $(E^\alpha, \varphi^\alpha)$ of $\G|_{U_\alpha}$ of length $\leq N$.
Assume that for any $\alpha\in I_j$, $j=1,2$, $E^\alpha$ is
equipped with a connection $D^\alpha$ which is compatible with
$(E^\alpha, \varphi^\alpha)$ in $U_\alpha \setminus \Sigma$,
and in case $\G = N\F$, then the connection $D_{-1}^\alpha$ on $N\F|_{U_\alpha}$
induced by $D_0^\alpha$ on $E_0$ is a basic connection on $U_\alpha \setminus \Sigma$.
Finally, for $j=1,2$, assume that $(\psi_\alpha)_{\alpha\in I_j}$ is
a partition of unity subordinate to $\U_j$.

Let $\Phi(D_j)$ denote the form \eqref{cardamom} associated to $D_j$ and $(\psi_\alpha)_{\alpha_j \in I_j}$.
Then there is a form $\eta^\Phi$ with support in $\Sigma$ such that
\begin{equation}\label{stenmark}
  d\eta ^\Phi = \Phi(D_1) -\Phi_1(D_2).
\end{equation}
The form $\eta^\Phi$ is locally given as a polynomial in the entries of $\theta^\alpha$,
$d\theta^\alpha$ and $\psi_\alpha$, where $\alpha \in I_1 \coprod I_2$,
and $\theta^\alpha$ is the connection matrix of $D^\alpha$ in some local frame.

Furthermore, if each $D^\alpha$ is of type $(1,0)$, then $\eta^\Phi$ is a sum of forms
of bidegree $(\ell+i,\ell-1-i)$ for $0\leq i \leq \ell-1$.
\end{lma}

  \begin{proof}
Let $\check{\eta}^\Phi \in \oplus_{p+q=2\ell-1} \check{C}^p(\U_1 \cap \U_2,\E^q)$
in Proposition~\ref{prop:Green-CdR-indep}, which satisfies \eqref{eq:nablaPotential1}.

For $j=1,2$, let $\Psi_{\U_j}:\check C^\bullet(\U_j, \E^\bullet)\to
\E^\bullet (M)$ be the morphism \eqref{cream} associated with the partition of unity $(\psi_\alpha)_{\alpha\in I_j}$
subordinate to $\U_j$.
Consider the refinement maps $r_j : \U_{12} \to \U_j$, $j=1,2$, from Proposition~\ref{prop:Green-CdR-indep},
where $\U_{12} = \U_1 \cap \U_2$,
and the partition of unity
\[
(\psi_{(\alpha, \beta)})_{(\alpha, \beta)\in I_1\times I_2},
~~~~~~~ \psi_{(\alpha, \beta)} = \psi_\alpha  \psi_\beta
\]
subordinate to $\U_{12}$. Let $\Psi_{\U_{12}} :  \check C^\bullet(\U_{12}, \E^\bullet)\to
\E^\bullet (M)$ be the associated
morphism \eqref{cream}.
Observe that (by Fubini)
      \begin{equation}\label{marsipan}
        \Psi_{\U_{12}} \circ r_j = \Psi_{\U_j}.
      \end{equation}
        Now let
        \[
          \eta^\Phi = \Psi_{\U_{12}} (\check \eta^\Phi).
        \]
Then
\begin{equation}\label{alcesalces}
    d\eta^\Phi=
    \Psi_{\U_{12}} (r_2 \check \Phi_{\U_2})- \Psi_{\U_{12}} (r_1 \check
    \Phi_{\U_1})
    =
    \Psi_{\U_2} (\check \Phi_{\U_2})-\Psi_{\U_1} (\check
    \Phi_{\U_1})
    =
    \Phi(D_2)- \Phi(D_1);
\end{equation}
here we have used that $\Psi_{\U_{12}}$ is a morphism of complexes and \eqref{kremla} for the first
equality, \eqref{marsipan} for the second equality, and
\eqref{cardamom} for the last equality.
This proves \eqref{stenmark}.

Recall that $\check \Phi_{\V}$ is defined by
\begin{equation}\label{anglosaxon}
(\check{\Phi}_{\V})_{\Delta} = \int_{\Delta_p} \Phi_{\V}(\widehat
D^{\Delta,\epsilon})
\end{equation}
if $\Delta = (\alpha_0,\ldots, \alpha_p)$,
see \eqref{rectangle}.
Recall from Lemma ~\ref{mercredi} that, for $\Delta \in N\V$,  $\Phi_{\V}(D^{\Delta})$ has
support in $(U_\Delta\cap \Sigma)\times \Delta_p$.
It follows that $(\check \Phi_{\V})_\Delta$ has support in $U_\Delta\cap
\Sigma$. Hence, for each $\Delta \in N\U_{12}$,
$(h\check \Phi_{\U_V})_\Delta$ has support in $U_\Delta\cap
\Sigma$.
Therefore $\eta^\Phi$ has
support in $\Sigma$.

By Proposition~\ref{prop:Green-CdR-indep}, the elements of $\check{\eta}^\Phi$
are polynomials in the entries $\theta^\alpha$ and $d\theta^\alpha$, and
thus, $\eta^\Phi$ is locally given as a polynomial in such entries and the
$\psi_\alpha$.

    \smallskip

Assume finally that for each $\alpha\in I_j$, $j=1,2$, that $D^\alpha$ is a $(1,0)$-connection.
Recall from the proof of Lemma~\ref{candy} that $\Psi$ maps an element in $\check C^p(\U, \E^{q_1, q_2})$ to a sum of forms of bidegree $(q_1+p_1, q_2+p_2)$, where
$p_1,p_2\geq 0$ and $p_1+p_2=p$.
It then follows from \eqref{eq:etaBidegree} that $\eta^\Phi$ is a sum of forms of bidegree $(\ell+i, \ell-1-i)$, where $0\leq i\leq \ell-1$.
\end{proof}

\section{Global Baum-Bott currents}\label{lisa}

Assume that $M$, $\G$, $Z$, $\Phi$, $\kappa$, $\U=(U_\alpha)_{\alpha \in I}$,
$(E^\alpha,\varphi^\alpha)$ are as in the previous section.
In this section, the resolutions $(E^\alpha,\varphi^\alpha)$ are assumed to be complexes
of holomorphic vector bundles.
Assume also that each $E^\alpha$ is equipped with Hermitian metrics and a connection $D^\alpha$.
We will apply the construction of connections $\tilde{D}_k$ from Section~\ref{sect:recallKLW}
on each $(E^\alpha,\varphi^\alpha),D^\alpha$, which yields connections defined on $U^\alpha \setminus Z$
satisfying certain properties.
If $\supp \G$ has codimension $\geq 1$, this is obtained from Lemma~\ref{lma:DtildeLW}
and in case $\G = N\F$, where $\F$ is a holomorphic foliation of $M$, this is obtained
from Lemma~\ref{lma:DtildeKLW}.
In both cases, this yields a family of connections $\tilde{D}^\alpha$ of the form
\begin{equation} \label{eq:tildeDcomparison}
    \tilde{D}_k^\alpha = D_k^\alpha + a_k^\alpha,
\end{equation}
where $a_k$ is smooth on $U_\alpha \setminus Z$, and where $a_k^\alpha$ extends as an
almost semi-meromorphic current to all of $U_\alpha$.

We now define a family of regularizations of those connections.
Let $\V=(V_\beta)_{\beta \in J}$ be an open cover of $M$.
Choose $\chi\sim \chi_{[1,\infty)}$ and on each $V_\beta$ choose a generically
nonvanishing holomorphic section $s^\beta$ of a Hermitian vector
bundle such that $Z \cap V_\beta \subset \{s^\beta=0\}$.
Let $(\psi_\beta)_{\beta\in J}$ be a partition of unity subordinate
to $\V$, and define a family of cut-off functions by
\begin{equation}\label{foggy}
  \chi_\epsilon = \sum_{\beta\in J}
 \psi_\beta \chi (|s^\beta|^2/\epsilon),
\end{equation}
and let $\Sigma_\epsilon$ be the closure
of $\{\chi_\epsilon<1\}$ in $M$.

Given connections of the form \eqref{eq:tildeDcomparison}, we define a new family of
connections $\widehat{D}^{\alpha,\epsilon}$ by
\begin{equation} \label{eq:DalphaRegularization}
\widehat{D}^{\alpha,\epsilon} := \chi_\epsilon \tilde{D}^\alpha + (1-\chi_\epsilon) D^\alpha = \chi_\epsilon a^\alpha + D^\alpha.
\end{equation}
Note that $\widehat{D}^{\alpha,\epsilon}$ is smooth on all of $U_\alpha$.

\begin{thm}\label{jelly}
Assume $M$ is a complex manifold of dimension $n$, that $\G$ is a coherent sheaf on
$M$, and that $Z$ denotes the set where $\G$ is not a vector bundle, and that either
$\supp \G$ has codimension $\geq 1$ and let $\kappa = n$,
or that $\G = N\F$, where $\F$ is a holomorphic foliation of rank $\kappa$ on $M$.
Let $\Phi\in \C[z_1,\ldots, z_n]$ be a homogeneous symmetric polynomial of
degree $\ell$ with $n-\kappa<\ell\leq n$.
Assume that $\U=(U_\alpha)_{\alpha \in  I}$ is a Stein open cover
of $M$ and $N \in \N$ be such that for each
$\alpha\in I$, there is a locally free resolution $(E^\alpha, \varphi^\alpha)$
of $\G|_{U_\alpha}$ of length $\leq N$.
In case $\G=N\F$, assume also that $E_0^\alpha = TM|_{U_\alpha}$.

Moreover, for each $\alpha\in I$, assume that $E^{\alpha}$ is equipped with Hermitian metrics
and a connections $D^{\alpha}$,
and if $\G=N\F$, assume that $D_0^{\alpha}$ is of type $(1,0)$ and torsion free.
Let $\chi_\epsilon$ be a cut-off function of the form \eqref{foggy},
and let   $(\psi_\alpha)_{\alpha\in I}$ be a partition of unity subordinate to
$\U$.
Let $(\widehat D^{\alpha,\epsilon})$ be the family
of connections defined in \eqref{eq:DalphaRegularization}.

Then
\begin{equation}\label{boring}
R^\Phi := \lim_{\epsilon\to 0} \Phi ( (\widehat D^{\alpha,\epsilon})_{\alpha \in I})
\end{equation}
is a well-defined closed pseudomeromorphic current of degree $2\ell$ with support on $Z$.
Moreover, $R^\Phi$ only depends on $(E^\alpha, \varphi^\alpha)_{\alpha\in I}$,
the Hermitian metrics and connections $D^\alpha$, and
$(\psi_\alpha)_{\alpha\in I}$ close to $Z$, and in particular it is independent of
the choice of $\chi_\epsilon$. If we assume that each
$D^\alpha$ is a $(1,0)$-connection, then $R^\Phi$
is a sum of currents of bidegree $(\ell+i, \ell-i)$ for $0\leq i\leq \ell$.
\end{thm}

Within the proof, we will consider connections whose entries in any local holomorphic frame
over some open set $U$ admit expansions of the form
\begin{equation}\label{pumpkin}
A+\chi_\epsilon B,
\end{equation}
where $A$ and $B$ are independent of $\chi_\epsilon$, $A$ is smooth $1$-form,
and $B$ is an almost semi-meromorphic current that is smooth outside of $Z$.
If one considers products of entries $C_j$ or $dC_j$, where $C_j$ is a $1$-form
of the form \eqref{pumpkin}, then we claim that one obtains an form which can be written as
\begin{equation} \label{gulogulo}
\alpha + \sum_{k \geq 1} \chi_\epsilon^k \beta_k + \sum_{k \geq 1} \chi_\epsilon^{k-1} d\chi_\epsilon \beta_k',
\end{equation}
where $\beta_k$ and $\beta_k'$ are almost semi-meromorphic. Indeed, this follows by using that almost
semi-meromorphic currents form an algebra, that $\1_{U\setminus Z} dB$ is almost semi-meromorphic
if $B$ is almost semi-meromorphic by Lemma~\ref{gold}, and that
$\chi_\epsilon \1_{U\setminus Z} dB = \chi_\epsilon dB$.

\begin{proof}
We will first show that the limit $R^\Phi$ exists as a
pseudomeromorphic current.
For $\alpha\in I$, let $\widehat \theta^{\alpha}_k$ be the connection
matrix of $\widehat D^{\alpha,\epsilon}_k$.
Then, in view of \eqref{eq:tildeDcomparison} and \eqref{eq:DalphaRegularization},
the entries of $\widehat \theta^{\alpha,\epsilon}_k$ in a local frame are of the form
\eqref{pumpkin}. If $\widehat{\theta}^{\Delta,\alpha,\epsilon}$ denotes the
connection matrix of the connection on $E^\Delta$ induced by $\widehat{D}^{\alpha,\epsilon}$
on $E^\alpha$ as defined by \eqref{eq:DDeltadef}, it follows that the entries
of $\widehat{D}^{\Delta,\alpha,\epsilon}$ are also of the same form.
It follows that for each $k$ the connection matrix of $\widehat D^{\epsilon,
\Delta}_k$ is of the form $\sum_{j=0}^p t_j \theta_j$, where
$\theta_j$ is a direct sum of forms \eqref{pumpkin}, so the entries of the curvature matrix
are polynomials in $t_1,\dots,t_p,dt_1,\dots,dt_p$ with coefficients of the form
\eqref{gulogulo}.
Thus, $\Phi(\widehat D^{\Delta,\epsilon} )$ may be expanded as a polynomial
of the same form.

Since each term in \eqref{gulogulo} is independent of $t$,
and linear combinations of terms of the form \eqref{gulogulo} are again
of this form, it follows that
$\check{\Phi}(\widehat D^{\epsilon} )_\Delta = \int_{\Delta_p} \Phi(\widehat D^{\Delta,\epsilon})$
is also of the form \eqref{gulogulo}.

By Proposition~\ref{prop:globalAsmRegularization} we conclude that the limit
as $\epsilon \to 0$ of terms of the form \eqref{gulogulo} exist as
pseudomeromorphic currents and the limits are independent of the
choice of $\chi$ and the $s^\alpha$. Since $\Phi (\widehat
D^{\epsilon})$ is a closed form of degree $2\ell$, $R^\Phi$ is a closed current of
degree $2\ell$.

Since $R^\Phi$ is independent of the  $s^\alpha$ we may choose $s^\alpha$
so that $\{s^\alpha=0\}= U_\alpha\cap S$. Indeed, if $\rho$ is the (generic) rank
of $\varphi_1^\alpha$, then we may take $s^\alpha = \bigwedge^\rho \varphi^\alpha_1$.
Then $\bigcap_{\epsilon>0} \Sigma_\epsilon = S$. By \eqref{eq:DalphaRegularization},
we have outside of $\Sigma_\epsilon$ that $\widehat{D}^{\alpha,\epsilon} = \widetilde{D}^\alpha$.
It then follows by Lemma~\ref{lma:DtildeLW}, Lemma~\ref{lma:DtildeKLW} and Lemma ~\ref{mercredi}
that $\Phi (\widehat D^{\epsilon})$ has support in $\Sigma_\epsilon$. Thus,
$R^\Phi$ has support in $Z$.

By the definitions of $\tilde{D}$ in Lemma~\ref{lma:DtildeLW} and Lemma~\ref{lma:DtildeKLW},
it follows that $\widehat D^{\alpha,\epsilon}$ are locally defined, in the sense that on
any open set $U \subset U_\alpha$, then $\widehat{D}^{\alpha,\epsilon}$ only
depend on $(E^\alpha,\varphi^\alpha)$, the Hermitian metrics on $E^\alpha$
and the connections $D^\alpha$ on $U$. It follows that also $\widehat D^{\Delta,\epsilon}$
is locally defined, and for any open set $U\subset U_\Delta$,
$\widehat D^{\Delta,\epsilon}$ only depends on the same data as above
for any $\alpha \in \Delta$.
Hence, $R^\Phi$ is locally defined
in the sense that it only depends on the $(E^\alpha,\varphi^\alpha)$, the Hermitian metrics and
connections $D^\alpha$, and $(\psi_\alpha)_{\alpha\in I}$ close to $Z$.

Assume that each $D^\alpha$ is a $(1,0)$-connection. Then, by Remark~\ref{axel}
and Lemma~\ref{lma:DtildeLW} or Lemma~\ref{lma:DtildeKLW}, $\widehat D^ {\alpha,\epsilon}$
is of type $(1,0)$. Hence, by Lemma ~\ref{candy}, $\Phi(\widehat D^\epsilon)$
is a sum of forms of bidegree $(\ell+i, \ell-i)$, $0\leq i\leq \ell$. It follows
that $R^\Phi$ is a sum of currents of bidegree $(\ell+i, \ell-i)$, $0\leq i\leq \ell$.
\end{proof}

\subsection{Independence of resolutions}

In this section we prove that the Baum-Bott currents in Theorem ~\ref{jelly}
only depend on $\Phi$ and $\mathcal{G}$, up to a $d$-exact
pseudomeromorphic current with support on $Z$.

\begin{thm}\label{pizza}
Let $M$ be a complex manifold of dimension $n$, let $\G$ be a coherent sheaf on
$M$, and let $Z$ denote the set where $\G$ is not a vector bundle. Assume that either
$\supp \G$ has codimension $\geq 1$ and let $\kappa = n$, or that $\G = N\F$, where
$\F$ is a holomorphic foliation of rank $\kappa$ on $M$.
Let $\Phi\in \C[z_1,\ldots, z_n]$ be a homogeneous symmetric polynomial of
degree $\ell$ with $n-\kappa<\ell\leq n$.

For $j=1,2$, assume that $\U_j=(U_\alpha)_{\alpha \in  I_j}$ is a Stein open cover
of $M$ and $N \in \N$ is such that for each $\alpha\in I_1 \coprod I_2$,
there is a locally free resolution $(E^\alpha, \varphi^\alpha)$ of $\G|_{U_\alpha}$
of length $\leq N$. In case $\G=N\F$, assume also that $E_0^\alpha = TM|_{U_\alpha}$.
Moreover, for each $\alpha\in I_1 \coprod I_2$, assume that $E^\alpha$
is equipped with Hermitian metrics and a connection $D^\alpha$, and if
$\G = N\F$, then assume that  $D_0^{\alpha}$ is of type $(1,0)$ and torsion free.

Let for $j=1,2$,  $(\psi_\alpha)_{\alpha\in I_j}$ be a partition of unity
subordinate to $\U_j$.
Let $R^\Phi_{(j)}, j=1,2$, denote the corresponding currents
\eqref{boring}.
Then there exists a pseudomeromorphic current $N^\Phi$ of degree
$2\ell -1$ with support on $Z$ such that
\begin{equation} \label{oberwolfach}
d N^\Phi=  R^\Phi_{(2)}-R^\Phi_{(1)}.
\end{equation}
Furthermore, if all $D^{\alpha}$ are of type $(1,0)$,
then $N^\Phi$ is a sum of currents of bidegree $(\ell+i,\ell-1-i)$ for $0\leq i \leq \ell-1$.
  \end{thm}

\begin{proof}
We may assume that the currents $R^\Phi_{(1)}$ and $R^\Phi_{(2)}$ are both defined with respect
to the same family of cut-off functions \eqref{foggy}.

We will use the notation from Proposition~\ref{prop:Green-CdR-indep} and Lemma ~\ref{hamburger}.
Let $N^\Phi_\epsilon$ be the form $\eta^\Phi$ defined in the proof of Lemma ~\ref{hamburger}
starting from the connections $\widehat D^{\alpha,\epsilon}$. By the proof of Proposition~\ref{prop:Green-CdR-indep}, $N^{\Phi}_\epsilon = h \check{\Phi}_\V$, where we use the notation from the
proof of Lemma ~\ref{hamburger} so that $\check \Phi_{\V}= \check \Phi_{\V}(\widehat
D^{\Delta,\epsilon})$

Recall from the proofs of Theorem ~\ref{jelly} and Lemma ~\ref{mercredi}
that, for $\Delta \in N\V$,
 $(\check \Phi_{\V})_\Delta$ is of the form \eqref{gulogulo} and has support in $U_\Delta\cap
\Sigma_\epsilon$.
Hence, for each $\Delta \in N\V$,
$(h\check \Phi_{\V})_\Delta$ is of the same form, cf.\ \eqref{semla}.
Therefore $N^\Phi_\epsilon$ is of the form \eqref{gulogulo} and has
support in $\Sigma_\epsilon$.
As in the proof of Theorem ~\ref{jelly} we conclude that the limit
\[
  N^\Phi := \lim_{\epsilon\to 0} N^\Phi_\epsilon
\]
is a well-defined pseudomeromorphic current that is independent of
the choice of $\chi_\epsilon$. Since the limit is independent of
the choice, we may locally assume that there is a section
$s$ such that $\{ s = 0 \} = Z$, and that $\chi_\epsilon = \chi(|s|^2/\epsilon)$.
Since for each $\epsilon >0$,  $N^\Phi_\epsilon$ has
support in $U_\alpha\cap \Sigma_\epsilon$, it follows that $N^\Phi$ has support
in $\bigcap_{\epsilon>0} \Sigma_\epsilon = Z$.

    Since \eqref{alcesalces} holds for $\epsilon>0$ with
    $\Phi_{\U_j}=\Phi_{\U_j}(\widehat D^\epsilon)$,
   it follows that
    \[
      dN^\Phi = \lim_{\epsilon\to 0}  \Phi_{\U_2}(\widehat D^\epsilon) -
      \lim_{\epsilon\to 0}
      \Phi_{\U_1}(\widehat D^\epsilon)
      =
      R^\Phi_{(2)}- R^\Phi_{(1)}.
  \]

 Assume now that each $D^\alpha$ is a $(1,0)$-connection.
By Lemma ~\ref{hamburger} $N^\Phi_\epsilon$ is
a sum of forms of bidegree $(\ell+i, \ell-1-i)$, where $0\leq i\leq
\ell-1$, and hence $N^\Phi$ is a sum of currents of the same bidegrees.
\end{proof}

\subsection{More precise versions of the main results}

The following theorem is a slightly more precise version of Theorem~\ref{thm:sheavesBasic}.

\begin{thm} \label{thm:sheavesGeneral}
Let $M$ be a complex manifold of dimension $n$, let $\G$ be a coherent analytic
sheaf on $M$ such that $Z = \supp \G$ has pure codimension $p \geq 1$ and let
$\Phi\in \C[z_1,\ldots, z_n]$ be a homogeneous symmetric polynomial of degree $\ell$
with $1 \leq \ell\leq n$.
Assume that $\U=(U_\alpha)_{\alpha \in  I}$ is a Stein open cover
of $M$ such that for each $\alpha\in I$, there is a locally free
resolution $(E^\alpha, \varphi^\alpha)$  of $\G|_{U_\alpha}$ of the form \eqref{anebrun}.
Moreover, for each $\alpha\in I$, assume that $E$ is equipped with Hermitian metrics
and connections $D^\alpha$.
Let $\chi_\epsilon$ be a cut-off function of the form \eqref{foggy}, and let
$(\psi_\alpha)_{\alpha\in I}$ be a partition of unity subordinate to
$\U$.
Let $\widehat D^{\alpha,\epsilon}$ be the family of connections defined in \eqref{eq:DalphaRegularization}.

Then
\begin{equation}
R^\Phi := \lim_{\epsilon\to 0} \Phi ( (\widehat D^{\alpha,\epsilon})_{\alpha \in I})
\end{equation}
is a well-defined closed pseudomeromorphic current of degree $2\ell$ with support on $Z$.
Moreover, $R^\Phi$ only depends on $(E^\alpha, \varphi^\alpha)_{\alpha\in I}$,
the Hermitian metrics and connections $D^\alpha$ and
$(\psi_\alpha)_{\alpha\in I}$ close to $Z$, and
in particular it is independent of the choice of $\chi_\epsilon$.

If we assume that all $D^\alpha$ are of type $(1,0)$,
then $R^\Phi$ is a sum of currents of bidegree
$(\ell+j, \ell-j)$ for $0\leq j\leq \ell$.
In addition,
\begin{equation} \label{eq:fundCycle'}
    R^{\elem_p} = (-1)^{p-1}(p-1)! [\G].
\end{equation}
If $\ell < p$, then
\begin{equation} \label{eq:vanish1'}
    R^{\elem_\ell} = 0,
\end{equation}
and if $\ell_1+\cdots+\ell_m \leq p$, where $m \geq 2$, then
\begin{equation} \label{eq:vanish2'}
    R^{\elem_{\ell_1} \cdots \elem_{\ell_m}} = 0.
\end{equation}
\end{thm}

\begin{proof}[Proof of Theorem~\ref{thm:sheavesGeneral}]
   The existence of $R^\Phi$, and the fact that it represents $\Phi(\mathcal{G})$,
   and has support on $Z$ follows immediately by Theorem~\ref{jelly}.

   It remains to prove the formulae \eqref{eq:fundCycle'}, \eqref{eq:vanish1'} and \eqref{eq:vanish2'}.
   Assume thus that $\Phi = e_{\ell_1} \cdots e_{\ell_m}$, where
   $\ell = \ell_1 + \dots + \ell_m \leq p$.
   Since the equalities we should prove are local statements, and $R^\Phi$ is locally
   defined, we may assume that all $U_\alpha$ are equal to $M$.
   By Theorem~\ref{pizza}, given different two choices of resolutions etc.,
   $R^\Phi$ only depends on these choices up to a current $dN^\Phi$, where
   $N^\Phi$ is a sum of pseudomeromorphic currents of bidegree $\leq \ell-1 \leq p-1$.
   By the dimension principle, $N^\Phi = 0$, so $R^\Phi$ is independent of all those choices.
   We may thus assume that $M$ is a Stein open set, equipped with a Stein open
   cover consisting of just $\U = (U_\alpha)_{\alpha \in pt} = \{ M \}$, with a resolution
   $(E,\varphi)$ and which is equipped with a connection $D$.
   If we let $\tilde{D}$ denote the connection on $E$ obtained by \eqref{lma:DtildeLW}
   and $\widehat{D}^\epsilon$ denote the connection obtained by the construction
   in \eqref{eq:DtildeRegularization} associated with $D,\tilde{D}$ and $\chi_\epsilon$,
   and $\widehat{D}^{\alpha,\epsilon}$ the induced connection on the cover,
   then clearly,
   \[
       \Phi(\widehat{D}^{\alpha,\epsilon}) = \Phi(\widehat{D}^\epsilon),
   \]
   and by \eqref{schmetterling},
   \[
    \Phi(\widehat{D}^\epsilon) = \elem_{\ell_1}(\widehat{D}^\epsilon) \wedge \dots
      \elem_{\ell_m}(\widehat{D}^\epsilon).
   \]
   The limit
   \[
       \lim_{\epsilon \to 0} \elem_{\ell_1}(\widehat{D}^\epsilon) \wedge \dots
         \elem_{\ell_m}(\widehat{D}^\epsilon)
   \]
   is exactly the current that in \cite{LW:chern-currents-sheaves} is denoted by
   $c^{Res}_{\ell_1}(E,D) \wedge \dots \wedge c^{Res}_{\ell_m}(E,D)$.

   The formulae we should prove then follow by \cite[equations (1.3)-(1.5)]{LW:chern-currents-sheaves}
\end{proof}

The following theorem is a slightly more precise version of Theorem~\ref{thm:foliationsBasic}.

\begin{thm} \label{thm:foliationsGeneral}
    Let $M$ be a complex manifold of dimension $n$, let $\F$ be a
    holomorphic foliation of rank $\kappa$ on $M$, and let $\Phi\in
    \C[z_1,\ldots, z_n]$ be a homogeneous symmetric polynomial of
    degree $\ell$ with $n-\kappa<\ell\leq n$.
    Assume that $\U=(U_\alpha)_{\alpha \in  I}$ is a Stein open cover
    of $M$ such that for each $\alpha\in I$, there is a locally free
    resolution $(E^\alpha, \varphi^\alpha)$  of $N\F|_{U_\alpha}$ of the form \eqref{anebrun}.
    Moreover, for each $\alpha\in I$, assume that $E^\alpha$ is
    equipped with Hermitian metrics and a connection $D^\alpha$
    such that $D_0^{\alpha}$ is of type $(1,0)$ and torsion free.
      Let $\chi_\epsilon$ be a cut-off function of the form \eqref{foggy},
      and let
  $(\psi_\alpha)_{\alpha\in I}$ be a partition of unity subordinate to
  $\U$.
  Let $\widehat D^{\alpha,\epsilon}$ be the family of connections defined in \eqref{eq:DalphaRegularization}.

  Then
  \begin{equation}
    R^\Phi := \lim_{\epsilon\to 0} \Phi ( (\widehat D^{\alpha,\epsilon})_{\alpha \in I})
  \end{equation}
  is a well-defined closed pseudomeromorphic current of degree $2\ell$
  with support on
  $\sing \F$.
  Moreover, $R^\Phi$ only depends on $(E^\alpha,
  \varphi^\alpha)_{\alpha\in I}$, the Hermitian metrics and connections $D^{\alpha}$,
  and $(\psi_\alpha)_{\alpha\in I}$ close to $\sing \F$, and
  in particular it is independent of the choice of $\chi_\epsilon$.

If $Z'$ is a compact connected component of $\sing \F$, then
\begin{equation} \label{motown}
    R^\Phi_{Z'} := \1_{Z'} R^\Phi
\end{equation}
is a pseudomeromorphic current with support on $Z'$ which represents the Baum-Bott residue
$\res^\Phi(\F;Z') \in H_{2n-2\ell}(Z',\C)$.

Assume now that each $D^\alpha$ is of type $(1,0)$.
Then $R^\Phi$ is a sum of currents of bidegree
  $(\ell+j, \ell-j)$ for $0\leq j\leq \ell$.
Furthermore, if $Z'$ is a connected component as above, and $\ell = \codim Z'$,
then $R^\Phi_{Z'}$ only depends on $\F$, and if $\ell < \codim Z'$, then $R^\Phi_{Z'} = 0$.
\end{thm}

\begin{proof}
The existence of $R^\Phi$, and the fact that it represents $\Phi(N\F)$,
and has support on $\sing \F$ follows immediately by Theorem~\ref{jelly}.

It remains to prove the properties of $R_{Z'}^\Phi$. Since $R^\Phi_{Z'}$ is locally defined,
we may replace $M$ by a neighborhood of $Z'$, and thus assume that $Z' = \sing \F = Z$.
Hence, $R^\Phi_{Z'} = R^\Phi$, which is pseudomeromorphic and has support on $Z$.

We now prove that $R^\Phi_{Z'}$ represents $\res^\Phi(\F; Z')$.
After possibly shrinking $M$ further, we may assume that $Z'$ is a deformation retract of $M$,
and that there is a locally free resolution $(E, \varphi)$ of the form \eqref{anebrun}
of $\mathcal A$-modules of $N\F \otimes \mathcal{A}$,
see Section ~\ref{lucia}.
Moreover, we may assume that the $E_k$ and $N\F|_{M\setminus Z}$ are equipped with
connections $\widehat{D}_k$ and $D_{basic}$ such that $D_{basic}$ is basic,
$\widehat{D}=(\widehat{D}_N,\ldots, \widehat{D}_0)$ is compatible with $(E, \varphi)$
in $M\setminus \Sigma$, for some compact neighborhood $Z\subset \Sigma\subset M$,
and $D_{basic}$ is the connection induced by $\widehat{D}$ on $M \setminus \Sigma$,
cf., Section ~\ref{bbres}.
Now the associated characteristic form $\Phi(\widehat{D})$ has support in $\Sigma$ and
represents $\res^\Phi(\F; Z)$ by definition, see Section ~\ref{lucia}.
By Example ~\ref{board} we may assume that $\Phi(\widehat{D})$ is defined from a
simplicial resolution with respect to a Stein open
cover.

Recall that $\Phi(\widehat D^\epsilon)$ has support in
$\Sigma_\epsilon$.
For $\epsilon > 0$ small enough, there is a neighborhood $U$ of $\Sigma$ such that
$\Sigma_\epsilon \cap U \subset \Sigma$. By replacing $M$ by $U$, we may thus assume
that $\Sigma_\epsilon \subset \Sigma$ for $\epsilon$ small enough.
We than obtain by Lemma ~\ref{hamburger} that
$\Phi(\widehat D^\epsilon)-\Phi(D) = d\eta^\epsilon,$
where $\eta^\epsilon$ is a form with support in $\Sigma$.
It follows that $\Phi(\widehat D^\epsilon)$ represents $\res^\Phi(\F;
Z)$ for $\epsilon>0$ small enough, see Section ~\ref{lucia}. Thus so does the limit $R^\Phi = R^\Phi_Z$
by Poincar\'e duality.

It remains to prove the properties of independence and vanishing of $R^\Phi_{Z'}$.
If $\ell < \codim Z'$, then since $R^\Phi_{Z'} = 0$ is a pseudomeromorphic current
of bidegree $(*,\ell)$ with support on $Z'$, it follows by the dimension principle
that $R^\Phi_{Z'} = 0$ if $\ell < \codim Z'$.

It remains to consider the case $\ell = \codim Z'$. Let $R^\Phi_{Z', (j)}$, $j=1,2$,
denote the Baum-Bott currents corresponding to two different choices of Stein
open cover, resolutions, Hermitian metrics, connections, and partition of unity.

Then, by Theorem ~\ref{pizza},
$$R^\Phi_{Z', (2)}- R^\Phi_{Z', (1)}=\1_{Z'} d N^\Phi,$$
where $N^\Phi$ is a pseudomeromorphic current with components of
bidegree $(*, q)$ with $q\leq \ell-1$, cf.\ \eqref{motown}.

Since $N^\Phi$ has support on $Z'$ of codimension $\geq \ell$, it
follows from the dimension principle, see Section ~\ref{hemma}, that
$N^\Phi$ vanishes, and consequently so does
$dN^\Phi$. Furthermore, since $\1_{Z'} d N^\Phi$ only depends on
$dN^\Phi$, it follows that $\1_{Z'} dN^\Phi$ vanishes. Thus
$R^\Phi_{Z', (1)}= R^\Phi_{Z', (2)}$, which proves the result.
\end{proof}

\bibliographystyle{alpha}
\bibliography{refs-BB}

\end{document}